\documentclass[11pt]{article}
\usepackage{}
\usepackage[total={6in, 8.5in}]{geometry}
\usepackage{amsfonts}
\usepackage{amsthm}
\usepackage{amssymb}
\usepackage{amsmath}
\usepackage{enumerate}
\usepackage[
pdfauthor={ESYZ},
pdftitle={Toughness and spanning trees in K4mf graphs},
pdfstartview=XYZ,
bookmarks=true,
colorlinks=true,
linkcolor=blue,
urlcolor=blue,
citecolor=blue,
bookmarks=false,
linktocpage=true,
hyperindex=true
]{hyperref}

\usepackage{tkz-graph}
\usetikzlibrary{arrows}
\usepackage{tkz-euclide}
\usetkzobj{all}

\newtheorem{THM}{Theorem}[section]
\newtheorem*{THM-k4mf-ftree}{Theorem \ref{k4mf-ftree}}
\newtheorem{LEM}[THM]{Lemma}
\newtheorem{PROP}[THM]{Proposition}
\newtheorem{OBS}[THM]{Observation}
\newtheorem{COR}[THM]{Corollary}
\newtheorem*{CLAu}{Claim} % Unnumbered Claim
\newenvironment{claimproof}[1]{\par\noindent\textit{Proof of claim.}\space#1}%
	{\hfill $\blacksquare$\smallbreak}

\theoremstyle{definition}
\newtheorem{QUES}[THM]{Question}

\newcommand{\CC}{\mathcal{C}}

\newcommand{\FF}{\mathcal{F}}

\def\AA{\mathcal{A}} % \newcommand won't work, override Angstrom!
\newcommand{\flow}{\varphi}
\newcommand{\cpc}{\gamma}
\newcommand{\rr}{r}
\def\ss{s} % \newcommand won't work, override German double s!
\newcommand{\xy}{\{x,y\}}
\newcommand{\uv}{\{u,v\}}
 \def\mZ{\mathbb{Z}}
 \def\zz{z}
 \let\de\delta
 \def\xx{\{x\}}
 \def\vv{\{v_1\}}
 \def\fft#1#2#3{\sum_{F \in \FF_#1(#3)} \omega_{#2}(F, #3)}
 \def\ffts#1{\fft{2}{}{#1}}

 \numberwithin{equation}{section}

\begin{document}
% Show labels in draft
\let\reallabel\label
\def\labelshow#1{\reallabel{#1}\relax\textnormal{\textbf{\footnotesize [#1]}}\quad}

\title{Toughness and spanning trees in $K_4$-minor-free graphs}

\author{%
 M. N. Ellingham\thanks{Supported by Simons Foundation award no. 429625.}%
\\
  Department of Mathematics, 1326 Stevenson Center,\\
  Vanderbilt University, Nashville, TN 37240\\
  \texttt{mark.ellingham@vanderbilt.edu}%
  \and
 Songling Shan\\
  Department of Mathematics,\\
  Illinois State University, Normal, IL 61790\\
    \texttt{sshan12@ilstu.edu}\and
 Dong Ye\thanks{Supported by Simons Foundation award no. 359516.}%
 \qquad Xiaoya Zha\\
  Department of Mathematical Sciences,\\
  Middle Tennessee State University, Murfreesboro, TN 37132\\
  \texttt{dong.ye@mtsu.edu}%
  \qquad
  \texttt{xiaoya.zha@mtsu.edu}%
}

\date{30 June 2019}
\maketitle

 \begin{abstract}
 For an integer $k$, a \emph{$k$-tree} is a tree with maximum degree at
most $k$.
 More generally, if $f$ is an integer-valued function on vertices, an
\emph{$f$-tree} is a tree in which each vertex $v$ has degree at most
$f(v)$.
 Let $c(G)$ denote the number of components of a graph $G$.
 We show that if $G$ is a connected $K_4$-minor-free graph and
 $$
 c(G-S) \;\le\; \sum_{v \in S} (f(v)-1)
 	\quad\hbox{for all $S \subseteq V(G)$ with $S \ne \emptyset$}
 $$
 then $G$ has a spanning $f$-tree.  Consequently, if $G$ is a
$\frac{1}{k-1}$-tough $K_4$-minor-free graph, then $G$ has a spanning
$k$-tree.
 These results are stronger than results for general graphs due to Win
(for $k$-trees) and Ellingham, Nam and Voss (for $f$-trees).
 The $K_4$-minor-free graphs form a subclass of planar graphs, and are
identical to graphs of treewidth at most $2$, and also to graphs whose
blocks are series-parallel.
 We provide examples to show that the inequality above cannot be relaxed
by adding $1$ to the right-hand side, and also to show that our result
does not hold for general planar graphs.
 Our proof uses a technique where we incorporate toughness-related
information into weights associated with vertices and cutsets.

 \end{abstract}

\vspace{2mm}

\section{Introduction}

All graphs considered are simple and finite. Let $G$ be a graph.
 We denote by $d_G(v)$ the degree of vertex $v$ in $G$.
 For $S\subseteq V(G)$ the subgraph induced on $V(G)-S$ is denoted by
$G-S$; we abbreviate $G-\{v\}$ to $G-v$.
 The number of components of $G$ is denoted by $c(G)$.
 The graph is said to be {\it $t$-tough} for a real number
$t \ge 0$ if $|S|\ge t\cdot c(G-S)$ for each $S\subseteq V(G)$ with $c(G-S)\ge
2$.
 The {\it toughness $\tau(G)$} is the largest real number $t$ for which
$G$ is $t$-tough, or $\infty$ if $G$ is complete.  Positive
toughness implies that
 $G$ is connected.  If $G$ has a hamiltonian cycle it is well known that
$G$ is $1$-tough.

 In 1973, Chv\'atal~\cite{chvatal-tough-c} conjectured that for some
constant $t_0$, every $t_0$-tough graph is hamiltonian.
 Thomassen (see \cite[p.~132]{STGT-ch6}) showed that there are
nonhamiltonian graphs with toughness greater than $3/2$.
	%(see~\cite{MR543656}).
 Enomoto, Jackson, Katerinis and Saito \cite{MR785651} showed that every
2-tough graph has a $2$-factor ($2$-regular spanning subgraph),
but also constructed $(2-\varepsilon)$-tough graphs with no $2$-factor,
and hence no hamiltonian cycle, for every $\varepsilon > 0$.
 Bauer, Broersma and Veldman \cite{Tough-CounterE} constructed
$(\frac{9}{4}-\varepsilon)$-tough nonhamiltonian graphs for every
$\varepsilon > 0$.  Thus, any such $t_0$ is at least $\frac{9}{4}$.

 There have been a number of papers on toughness conditions that
guarantee the existence of more general spanning structures in a graph.
 A {\it $k$-tree} is a tree with maximum degree at most $k$, and a {\it
$k$-walk} is a closed walk with each vertex repeated at most $k$
times.
 Note that a spanning 2-tree is a hamiltonian path and a spanning 1-walk
is a hamiltonian cycle.
 Jackson and Wormald \cite{JW-k-walks} showed that on a given vertex set
a $k$-walk can be obtained from a $k$-tree; conversely, a $(k+1)$-tree
can be obtained from a $k$-walk.
  More generally, if $f: V(G) \to \mZ$ then an \emph{$f$-tree} is a tree
with $d_T(v) \le f(v)$ for all $v \in V(T)$, and an \emph{$f$-walk} is a
closed walk that uses every vertex $v$ at most $f(v)$ times.

 The first toughness result for spanning trees of bounded degree was by
Win.

 \begin{THM}[Win \cite{Win-tough}]\label{Win-ktree}
 Suppose $G$ is a connected graph, $k \ge 2$, and
 \begin{eqnarray*}
 c(G-S) &\le& (k-2)|S| + 2
	\quad\hbox{for all $S \subseteq V(G)$}.
 \end{eqnarray*}
 Then $G$ has a spanning $k$-tree.
 \end{THM}

 \noindent
  Win's result implies that for $k \ge 3$, $\frac{1}{k-2}$-tough graphs have a
$k$-tree (and hence a $k$-walk).  Ellingham, Nam and Voss showed that
the bound on the degrees need not be constant.

 \begin{THM}[Ellingham, Nam and Voss \cite{ENV}]\label{ENV-ftree}
 Suppose $G$ is a connected graph, $f: V(G) \to \mZ$ with $f(v) \ge 2$
for all $v \in V(G)$, and
 \begin{eqnarray}\label{suf-ftree}
 c(G-S) &\le& \sum_{v \in S} (f(v)-2) + 2
	\quad\hbox{for all $S \subseteq V(G)$} .
 \end{eqnarray}
 Then $G$ has a spanning $f$-tree.
 \end{THM}

 While these conditions are sufficient, and sharp in the sense that the
right-hand side of the inequality cannot be increased by $1$, they are
not necessary.  For a necessary condition, a graph with a spanning
$k$-walk (and hence a graph with a spanning $k$-tree) must be
$\frac{1}{k}$-tough.
 However, a stronger necessary condition, which also applies for
non-constant degree bounds, can be obtained by counting components in a
tree $T$: for any $S \subseteq V(T)$, $c(T-S) \le \sum_{v \in S}
(d_T(v)-1) + 1$.  Applying this to a spanning $f$-tree of a graph $G$
gives the following.

 \begin{OBS}
 Suppose $G$ is a graph with a spanning $f$-tree.  Then
 \begin{eqnarray}\label{nec-ftree}
 c(G-S) &\le& \sum_{v \in S} (f(v)-1) + 1
	\quad\hbox{for all $S \subseteq V(G)$}.
 \end{eqnarray}
 \end{OBS}

 %  % \tag requires amsmath package and does not work with eqnarray
 %  % add  \;\; to make spacing more like eqnarray
 %  \begin{OBS}
 %  Suppose $G$ is a graph with a spanning $f$-tree.  Then
 %  \begin{equation}
 %  c(G-S) \;\;\le\;\; \sum_{v \in S} (f(v)-1) + 1
 % 	\quad\hbox{for all $S \subseteq V(G)$}. \tag{AA}
 %  \end{equation}
 %  \end{OBS}

 \noindent
 When $f(v)=k$ for all $v$, this condition is slightly weaker than being
$\frac{1}{k-1}$-tough.

 Recently some stronger versions of Theorem \ref{ENV-ftree} have been
posted by Hasanvand \cite[Subsection 4.2]{Ha17v5}, although they do not
seem to change the condition (\ref{suf-ftree}) in a substantial way.
 Hasanvand does, however, appear to have made a significant advance for
$f$-walks.
 For $k \ge 3$, as noted above, $\frac{1}{k-2}$-tough graphs have a
spanning $k$-tree and hence a spanning $k$-walk.  Jackson and Wormald
\cite{JW-k-walks} conjectured that in fact $\frac{1}{k-1}$-tough graphs
have a spanning $k$-walk.
 Hasanvand used a clever argument to combine a stronger form of Theorem
\ref{ENV-ftree} with a result of Kano, Katona and Szab\'{o} \cite{KKS09}
on the existence of spanning subgraphs with parity conditions, to show
the following.

 \begin{THM}[Hasanvand, {\cite[Theorem 5.5]{Ha17v5}}]\label{Ha-fwalk}
 Suppose $G$ is a graph, $M$ is a matching in $G$, $f: V(G) \to \mZ$
with $f(v) \ge 1$ for all $v \in V(G)$, and
 \begin{equation*}
 c(G-S) \;\;\le\;\; \sum_{v \in S} (f(v)-1) + 1
	\quad\hbox{for all $S \subseteq V(G)$}.
	\tag*{(\ref{nec-ftree}) again}
 \end{equation*}
 Then $G$ has a spanning $f$-walk that uses the edges of $M$.
 \end{THM}

 \noindent
 % The condition here is identical to (\ref{nec-ftree}).
 With $f(v)=k$ for all $v$, Theorem \ref{Ha-fwalk} verifies Jackson and
Wormald's conjecture.

 It is unknown whether Theorem \ref{ENV-ftree} can be improved for
graphs in general by weakening condition (\ref{suf-ftree}).
 It is therefore of interest to see whether this can be done for special
classes of graphs.
 In this paper we focus on $K_4$-minor-free graphs.

 A graph $H$ is a {\it minor} of a graph $G$ if a graph isomorphic to
$H$ can be obtained from $G$ by edge contractions, edge deletions and
vertex deletions; if not, $G$ is \emph{$H$-minor-free}.
 Since both  $K_{3,3}$ and $K_5$ contain a $K_4$ minor, $K_4$-minor-free
graphs are $K_{3,3}$-minor-free and $K_5$-minor-free, i.e., planar.
 The class of $K_4$-minor-free graphs includes
all \emph{series-parallel
 graphs}, constructed by
 series and parallel compositions starting from copies of $K_2$.
 Duffin \cite{MR0175809} gave three characterizations for
series-parallel graphs; in particular, he showed that
 a graph with no cutvertex is $K_4$-minor-free if and only if it is
series-parallel.
	% This is Theorem 2 in Duffin's paper.
 Wald and Colbourn \cite{WC-part2tree} showed that
$K_4$-minor-free graphs are also identical to graphs of treewidth at
most $2$.

 Let $G$ be a $K_4$-minor-free graph with toughness greater than
$\frac{2}{3}$.
 The toughness implies that $G$ has no cutvertex, and in conjunction
with $K_4$-minor-freeness, also implies that $G$ is $K_{2,3}$-minor-free
(see Lemma \ref{k4-minor-free-no-N2C}). Thus, $G$ is either
$K_2$ or a $2$-connected outerplanar graph.  Thus, a $K_4$-minor-free
graph with at least three vertices and toughness greater than
$\frac{2}{3}$ is hamiltonian.
 Dvo{\v{r}}{\'a}k, Kr\'{a}l' and Teska \cite{K4MF-no-2-walk} showed that
every $K_4$-minor-free graph with toughness greater than $\frac{4}{7}$
has a spanning  $2$-walk, and they constructed a $\frac{4}{7}$-tough
$K_4$-minor-free graph with no spanning 2-walk.

 Our main result is a sufficient condition for a connected
$K_4$-minor-free graph to have an $f$-tree.  We show that in fact a very
slight strengthening of the necessary condition (\ref{nec-ftree})
suffices, subtracting $1$ from the right-hand side.
 For vertices $x, v$ in a graph, define $\de_x(v)$ to be $1$ if $x=v$,
and $0$ otherwise.

 % Define theorem text part as macros, so can ensure identical later
 % when restate the theorem.
 \global\def\mainhypothesis{Let $G$ be a connected $K_4$-minor-free
graph, and $f : V(G) \to \mZ$.  Suppose that $\zz \in V(G)$ and}
 \global\def\mainconclusion{Then $G$ has a spanning $(f-\de_\zz)$-tree,
i.e., a spanning $f$-tree $T$ such that $d_T(\zz)\le f(\zz)-1$.}

 \begin{THM}\label{k4mf-ftree}
 \mainhypothesis
 \begin{equation}\label{TS}
 c(G-S) \;\;\le\;\; \sum_{v\in S}(f(v)-1)
	\quad\hbox{for all $S\subseteq V(G)$ with $S\ne \emptyset$.}
	\tag{CC}
 \end{equation}
 \mainconclusion
 \end{THM}

  The condition in
Theorem \ref{k4mf-ftree} is something we refer to frequently; we call it
the \emph{Component Condition} and refer to it as (\ref{TS}).
 Condition (\ref{TS}) with $S = \{v\}$ implies that $f(v) \ge 1$ if
$V(G) = \{v\}$, and that $f(v) \ge 2$ for all $v \in V(G)$ if $|V(G)|
\ge 2$.
 Therefore, we do not need to explicitly specify that $f$ is
nonnegative.  If we take $f(v)=k\ge 2$ for all the vertices, then
(\ref{TS}) just means that $G$ is $\frac{1}{k-1}$-tough.

\begin{COR}\label{k4mf-ktree}
 Let $G$ be a $K_4$-minor-free graph, let $k \ge 2$ be an
integer, and let $\zz \in V(G)$.
 If $G$ is $\frac{1}{k-1}$-tough, then $G$ has a spanning $k$-tree $T$
such that $d_T(\zz)\le k-1$.
 \end{COR}

 Examples show that we cannot weaken (\ref{TS}) in Theorem
\ref{k4mf-ftree} by adding $1$ to the right-hand side, which gives
(\ref{nec-ftree}). Let $f(x)\ge2, f(y)\ge 2, f(z)\ge 2$ be three
integers, and let $G$ be the $K_4$-minor-free graph obtained from a
triangle $(xyz)$ by adding $f(v)-1$ pendant edges at each $v\in
\{x,y,z\}$.   Set $f(v)=2$ for every $v\in V(G)-\{x,y,z\}$.
 Then (\ref{nec-ftree}) is easily verified, but every
spanning tree $T$ of $G$ has a vertex $v\in \{x,y,z\}$ with
$d_T(v)=f(v)+1$.

 Also, note that there is no upper bound on the values we can assign to
$f$.  If we take a specified set of vertices $X$ and let $f(v) =
|V(G)|+1$ for all $v \in V(G)-X$, (\ref{TS}) is automatically satisfied
for sets $S$ containing a vertex of $V(G)-X$ (we use this trick in some
of our proofs).
 We can therefore bound the degrees of just the specified vertices.

 \begin{COR}\label{k4mf-ftree-spec}
 Let $G$ be a connected $K_4$-minor-free graph, $X \subseteq V(G)$, and
$f : X \to \mZ$.  Suppose that $\zz \in X$ and
 \begin{equation*}
 c(G-S) \;\;\le\;\; \sum_{v\in S}(f(v)-1)
	\quad\hbox{for every $S\subseteq X$ with $S\ne \emptyset$.}
 \end{equation*}
 Then $G$ has a spanning tree $T$ such that $d_T(v) \le f(v)$ for all $v
\in X$, and $d_T(\zz)\le f(\zz)-1$.
 \end{COR}

 Toughness is an awkward parameter to deal with in arguments.
It is hard to control the toughness of a subgraph, or of a graph obtained by
some kind of reduction from an original graph.  This makes it difficult
to prove results based on toughness conditions using induction.
 However, in this paper we provide a way of doing induction using a
toughness-related condition, by first, in Section
\ref{sec-n2c}, transforming the toughness information into weights
associated with vertices and cutsets.
 This approach seems to be new, and of interest apart from our results
for $K_4$-minor-free graphs.

 Another idea in this paper is dealing with situations where we delete
independent sets of vertices separately from situations where we delete
sets that may have adjacencies.  Loosely, we expect to get more
components when we delete an independent set.  The concept of
`fundamental' graphs in Section \ref{sec-n2c} is used to implement this
distinction.  In Section \ref{proof-k4mf-ftree} we prove a result for
fundamental graphs first, then extend it to more general graphs.

\section{Structure of nontrivial 2-cuts }
\label{sec-n2c}

 In this section we show how to convert a toughness-related condition
into weights associated with vertices and certain cutsets.
 The results in this section apply to all graphs, not just to
$K_4$-minor-free graphs.

 If $G$ is a graph and $H \subseteq G$ ($H$
may just be a set of vertices) then a \emph{bridge of
$H$} or \emph{$H$-bridge} in $G$ is a
subgraph of $G$ that is
 either an edge not in $H$ joining two vertices of
$H$ (a \emph{trivial} bridge), or a component of
$G-V(H)$
together with all edges joining it to $V(H)$
 (a \emph{nontrivial} bridge).
 The \emph{set of attachments} of an $H$-bridge $B$ is
$V(B) \cap V(H)$. For $S \subseteq
V(G)$, when no confusion will result, we also refer to the set of
attachments of a component of $G-S$, meaning that of the corresponding
$S$-bridge.
 A \emph{$k$-cut} in $G$ is a set $S$ of $k$ vertices for which $G-S$ is
disconnected.
 If $F = \{u, v\}$ is a set of two vertices in $G$, let $c(G,F)$ or
$c(G, uv)$ denote the number of nontrivial $F$-bridges that contain both
vertices of $F$.  The notation $c(G, uv)$ does not imply that $uv \in
E(G)$.
 A \emph{nontrivial $2$-cut} or \emph{N2C} in $G$ is a set $F$ of two
vertices with $c(G, F) \ge 3$.

 A \emph{block} is a connected graph with no cutvertex, and a
 \emph{block of $G$} is a maximal subgraph of $G$ that is itself
a block. Let $\mathcal{B}$ be the set of blocks and $\mathcal{C}$ the
set of cutvertices of $G$. The {\it block-cutvertex tree}  of a
connected graph $G$ has vertex set $\mathcal{B}\cup \mathcal{C}$, and
$c\in \mathcal{C} $ is adjacent to $B\in \mathcal{B}$  if and only if
the block $B$ contains the cutvertex $c$.
 A block of $G$ is a \emph{leaf block} if it is a leaf of the
block-cutvertex tree.
 If we choose a particular \emph{root edge} $e_0$ of $G$,
then the block $B_0$ containing $e_0$ is the \emph{root block} which we
treat as the root vertex of the block-cutvertex tree.
 Every block other than $B_0$ has a unique \emph{root vertex}, namely
its parent in the rooted block-cutvertex tree.

 In what follows we try to develop information to help us construct a
spanning tree with restricted degrees, by allocating the bridges of each
nontrivial $2$-cut $\{u, v\}$ between $u$ and $v$.  We begin with some
basic properties of nontrivial $2$-cuts.

\begin{LEM}\label{SP-tree2}
 Let $\{u,v\}$ be an N2C in a graph $G$.

 \noindent
 \textnormal{(a)} The graph $G$ has three internally
disjoint $uv$-paths.

 \noindent
 \textnormal{(b)} If $S \subseteq V(G)-\{u,v\}$ and
$|S|\le 2$ then $u$ and $v$ lie in the same component of $G-S$.

 \noindent
 \textnormal{(c)} If $\{x,y\}$ is an N2C of $G$
distinct from $\{u,v\}$, then $u$ and $v$ lie in a unique
$\{x,y\}$-bridge.
 \end{LEM}

 \begin{proof}
 Since $c(G,uv) \ge 3$, there are three different $\{u,v\}$-bridges
that attach at both $u$ and $v$ and we can take a path through
each, proving (a).  Then (b) follows immediately.  For (c) we may assume
that $u \notin \{x,y\}$.  Then the unique $\{x,y\}$-bridge containing
$u$ also contains $v$, either because $v \in \{x,y\}$, or by (b)
otherwise.
 \end{proof}

 Let $G$ be a connected graph,  $\FF$ the set of N2Cs of $G$,  and
$W=\bigcup _{F\in \FF} F $.
 The graph with vertex set $W$ and edge set $\FF$ is called the
 \emph{N2C graph of $G$}.
  If $W$ is an independent set in $G$, we call $G$ \emph{fundamental}.

 The following observations come from comparing the components of $G-S$
and of $G'-S$ for some subgraph (block or union of
$\{x,y\}$-bridges) $G'$ of $G$, and using Lemma \ref{SP-tree2}.

 \begin{OBS}\label{n2cblock}
 Suppose $G$ is a connected graph with $|V(G)| \ge 2$.
 Let $u, v \in V(G)$.
 If $u$ and $v$ are not in a common block of $G$, then $c(G, uv) = 1$; if they belong to a common block $B$ then
$c(B, uv) = c(G, uv)$.

 Thus, the N2Cs of a block $B$ are the N2Cs $\{u,v\}$ of $G$ with
$\{u,v\} \subseteq V(B)$, and if $G$ is fundamental then $B$ is also
fundamental.
 Conversely, the set of N2Cs of $G$ is the disjoint union of the sets of
N2Cs of its blocks, and if each block is fundamental then $G$ is
fundamental.
 \end{OBS}

 \begin{OBS}\label{n2cbridge}
 Suppose $\xy$ is an N2C of a $2$-connected graph $G$.
 Let $G'$ be the union of two or more nontrivial $\xy$-bridges of $G$
(possibly $G'=G$).
 Then $G'$ is $2$-connected.
 If $u$ and $v$ are not in a common $\xy$-bridge in $G'$, then
 $c(G, uv)=1$; if they belong to a common
$\xy$-bridge in $G'$ and $\uv \ne \xy$ then $c(G',uv) = c(G, uv)$.

 Thus, the N2Cs of $G'$ are the N2Cs $\uv$ of $G$ with $\uv \subseteq
V(D)$ for some $\xy$-bridge $D$ of $G'$, except that $\xy$ is not
necessarily an N2C of $G'$.
 If $G$ is fundamental then $G'$ is also fundamental.
 \end{OBS}

 Observation \ref{n2cbridge} is not necessarily true if we take $G'$ to
be a single nontrivial $\xy$-bridge; then $G'$ may not be $2$-connected,
and we have $c(G',uv) > c(G,uv)$ if $x$ and $y$ are in different
components of $G'-\uv$.

 \begin{LEM}\label{N2Csubgraph}
 Suppose $J$ is a connected subgraph of the N2C graph $H$ of a graph
$G$.

 \noindent
 \textnormal{(a)} If $S \subseteq V(G)$, $|S| \le 2$ and $S \cap V(J) =
\emptyset$, then $V(J)$ lies in a single component of $G-S$.

 \noindent
 \textnormal{(b)} If $J'$ is a connected subgraph of $H$
vertex-disjoint from $J$, then $V(J)$ lies in a single component of
$G-V(J')$.

 \noindent
 \textnormal{(c)} If $u, v \in V(J)$ are separated in $G$ by a cutvertex
$x$ of $G$, then $x$ is also a cutvertex of $J$ separating $u$ and $v$
in $J$.
 %
 % NOT NEEDED:
 %  \noindent
 %  \textnormal{(d)} For each $uv \in E(J)$ there is a unique block of $G$
 % that contains both $u$ and $v$.
 \end{LEM}

 \begin{proof}
 (a) By Lemma \ref{SP-tree2}(b), for each $uv \in E(J)$, $u$ and $v$
are in the same component of $G-S$.  Applying this repeatedly, all of
$V(J)$ is in a single component of $G-S$.

 (b) Let $uv \in E(J)$.  By (a), all of $V(J')$ lies in the same
component of $G-\{u,v\}$, say $C$.  Since $\{u,v\}$ is an N2C of $G$,
there is a path from $u$ to $v$ in
$G-V(C) \subseteq G-V(J')$, so $u$
and $v$ are in the same component of $G-V(J')$.  Applying this
repeatedly, all of $V(J)$ is in a single component of $G-V(J')$.

 (c) If there is a $uv$-path in $J$ that does not contain $x$, then by
(a) this path lies in a single component of $G-x$, contradicting the
choice of $u$ and $v$.  Thus, every $uv$-path in $J$ contains $x$, so
$x$ is a cutvertex separating $u$ and $v$ in $J$.
 %
 %  (d) If $u$ and $v$ are not in the same block there is a cutvertex $x$
 % that separates them, but then $S=\{x\}$ violates (a).
 %  There cannot be two blocks that both contain $u$ and $v$.
 \end{proof}

 % *** \cmb\cmn{Would be good to have proper environment for cases.}

 Our overall strategy now is to use our toughness-related condition to
assign weights associated with (N2C, vertex) ordered pairs.  We show
that we can either (a) assign weights satisfying certain conditions, or
else (b) find a set that violates our toughness-related
condition.
 The following proposition, giving a lower bound on
$c(G-U)$ for certain subsets $U$, will be used to demonstrate (b).
 The statement of this result and many of our
computations involve terms of the form $c(G-u)-1$ and $c(G, uv)-2$.  The
reader may wonder why we do not simplify these; the answer is that we
often think of the graph as having a `main' part consisting of one of
the bridges of a cutvertex, or two of the bridges of a $2$-cut, and
these terms count the number of `extra' bridges outside the `main' part.

\begin{PROP}\label{com-counting}
 Let $J$ be a connected subgraph of the N2C graph of a connected
fundamental graph $G$, and $U = V(J)$.  Then
 \begin{eqnarray}\label{comp-count}
  c(G-U)\ge \sum_{uv\in E(J)}(c(G,uv)-2)+\sum_{w\in U}(c(G-w)-1)+|U|.
 \end{eqnarray}
\end{PROP}

\begin{proof}
 The proof is by induction on $|E(J)|$.  It is easy to check that the
conclusion is true if $|E(J)|=0$, when $|U|=|V(J)|=1$.  So we assume
that $|E(J)|\ge 1$.
  We consider three cases.

 \smallskip\noindent
 \textbf{Case 1:} Suppose $G$ has no cutvertex.  Since
$|E(J)|\ge 1$, $G$ has at least five vertices, so $G$ is $2$-connected.
For the $2$-connected case (\ref{comp-count}) simplifies to
 \begin{eqnarray}\label{comp-count-2c}
  c(G-U)\ge \sum_{uv\in E(J)}(c(G,uv)-2)+|U|.
 \end{eqnarray}
 Let $xy \in E(J)$ and let the $\{x,y\}$-bridges in $G$ be $D_1, D_2,
\dots, D_k$, where $k = c(G,uv) \ge 3$.
 For each $i$ let $U_i = U \cap V(D_i) \supseteq \{x,y\}$ and
$J_i = J[U_i] - xy$, so that $V(J_i)
= U_i$.

 Let $G_i = D_i \cup xa_iy$, where $a_i$ is a new vertex.
 Each $G_i$ is a minor of $G$, and hence $K_4$-minor-free.
 We may apply Observation \ref{n2cbridge} to both $G_i$ and $G$,
considered as subgraphs of $G \cup xa_i y$.
 We see that $G_i$ is $2$-connected.
 Suppose that $\uv
\subseteq V(G_i)$ and $\uv \ne \xy$.  If $a_i \in \uv$, then $c(G_i,
uv)=1$.  If $a_i \notin \uv$ then $c(G_i, uv) = c(G \cup xa_iy, uv) =
c(G,uv)$.  Thus, $\uv \ne \xy$ is an N2C of $G_i$ if and only if it is
an N2C of $G$ with $\uv \subseteq V(G_i)$.  It follows that $G_i$ is
fundamental.
 Moreover, if $H_i$ is the N2C graph of $G_i$, we have $E(J_i) \subseteq
E(H_i)$ and $V(J_i) \subseteq V(H_i) \cup \{x,y\}$ (each of $x$ and $y$
is in $J_i$, but could possibly be in no N2C of $G_i$, in which case it
is an isolated vertex of $J_i$ and not a vertex of $H_i$).
 Since $xy \notin E(J_i)$, $|E(J_i)| < |E(J)|$ for all $i$.
 (In future we will provide less detail when applying Observation
\ref{n2cbridge}, but the reasoning will be similar to the reasoning
here.)

 \begin{CLAu}\label{claimcount}
 We have
 $\displaystyle c(G_i - U_i) \ge \sum_{uv \in E(J_i)} (c(G_i, uv) - 2) + |U_i|$ for each
$i=1,2,\dots k$.
 \end{CLAu}

 \begin{claimproof}
 First suppose that $J_i$ is connected.  Then $x$ and $y$ are incident
with edges of $J_i$, so they are contained in N2Cs of $G_i$, and so $J_i
\subseteq H_i$.  Thus, the claim follows because we can apply
(\ref{comp-count-2c}) to $G_i-U_i$ and $J_i$ by induction.

 Now suppose that $J_i$ is disconnected.  Since $J$ is connected, $J_i$
must have exactly two components $J_i^x$ and $J_i^y$, containing $x$ and
$y$ respectively.  Let $U_i^x = V(J_i^x)$ and $U_i^y = V(J_i^y)$.
 Consider the components of $G_i-U_i$ and divide them into three groups:
$\CC_x$ are those all of whose attachments belong to $U_i^x$, $\CC_y$ are
those all of whose attachments belong to $U_i^y$, and $\CC_{xy}$ are those
that have attachments in both $U_i^x$ and $U_i^y$.

 We claim that $|\CC_{xy}| \ge 2$.  Obviously $a_i \in \CC_{xy}$.
 Take an $xy$-path $v_0 v_1 \dots v_t$ in $D_i$.  There is a subpath
$P=v_i v_{i+1} \dots v_j$ such that $v_i \in U_i^x$, $v_j \in U_i^y$ and
the internal vertices (if any) of $P$ belong to neither $U_i^x$ nor
$U_i^y$.  Since $G$ is fundamental, $P$ does have internal vertices,
which belong to a second component in $\CC_{xy}$.

 Now we claim that
 \begin{eqnarray}\label{ineqcx}
 \displaystyle |\CC_x| \ge \sum_{uv\in E(J_i^x)}(c(G_i,uv)-2)+|U_i^x|-1.
 \end{eqnarray}
 If $|U_i^x|=1$ this just says that $|\CC_x| \ge 0$, which is
trivially true.  So we may suppose that $|U_i^x| \ge 2$, which means
that $x$ is in an N2C of $G_i$ and $J_i^x \subseteq H_i$.  Then $U_i^y$ is in a single component of
$G_i-U_i^x$, either trivially if $|U_i^y|=1$, or by Lemma
\ref{N2Csubgraph}(b) if $|U_i^y| \ge 2$ so that $J_i^y$ is a subgraph of
$H_i$.
 The component of $G_i-U_i^x$ containing $U_i^y$ must also contain all
components in $\CC_y \cup \CC_{xy}$.
 On the other hand, each $C \in \CC_x$ is a separate component of $G_i -
U_i^x$.  Therefore, $c(G_i-U_i^x) = |\CC_x|+1$.  Using this and
applying (\ref{comp-count-2c}) to $G_i-U_i^x$ and $J_i^x$ by induction
gives  (\ref{ineqcx}).

 We have an inequality similar to (\ref{ineqcx}) for $\CC_y$.
Therefore,
 \begin{eqnarray*}
 c(G_i-U_i) &=& |\CC_x| + |\CC_y| + |\CC_{xy}| \\
  &\ge&
 \sum_{uv\in E(J_i^x)}(c(G_i,uv)-2)+|U_i^x|-1 + \\
 &&\qquad\qquad \sum_{uv\in E(J_i^y)}(c(G_i,uv)-2)+|U_i^y|-1 + 2 \\
  &=& \sum_{uv \in E(J_i)} (c(G_i, uv) - 2) + |U_i|
 \end{eqnarray*}
 which proves the claim.
 \end{claimproof}

 Now we prove (\ref{comp-count-2c}) and hence (\ref{comp-count}).  For
each $i$, $c(D_i-U_i) = c(G_i-U_i)-1$ since we lose the component $a_i$.
 Thus,
 \begin{eqnarray*}
 c(G-U) &=& \sum_{i=1}^k c(D_i-U_i) \;\;=\;\; \sum_{i=1}^k (c(G_i-U_i)-1) \\
	&\ge& \sum_{i=1}^k \left( \sum_{uv \in E(J_i)} (c(G,uv)-2)
		+ |U_i| - 1 \right) \\
	&&\qquad\qquad\hbox{by the claim and since $c(G_i,uv)=c(G,uv)$} \\
	&=&  \sum_{uv \in E(J-xy)} (c(G,uv)-2) + |U| + 2(k-1) - k \\
	&&\qquad\qquad\hbox{since $x$ and $y$ are overcounted $k-1$ times each} \\
	&=&  \sum_{uv \in E(J-xy)} (c(G,uv)-2) + |U| + c(G,xy)-2 \\
	&=&  \sum_{uv \in E(J)} (c(G,uv)-2) + |U|.
 \end{eqnarray*}

 \smallbreak\noindent
 \textbf{Case 2:} Suppose $G$ has a cutvertex but all vertices of $J$
lie in a single block $B$ of $G$.
 If $S \subseteq U$ then each component of $G-S$ is either (a) a
component $C$ of $B-S$ together with all $B$-bridges in $G$ that attach
at a vertex of $C$, or (b) a component of $G-w$ not containing $B-w$, for some $w \in S$.
Thus,
 \begin{eqnarray}\label{c2a}
 c(G-U) = c(B-U) + \sum_{w \in U} (c(G-w)-1) .
 \end{eqnarray}
 By Observation \ref{n2cblock}, $B$ is fundamental, $J$ is a subgraph of
the N2C graph of $B$, and
 \begin{eqnarray}\label{c2c}
 c(G,uv) = c(B,uv) \quad\hbox{for all $uv \in E(J)$} .
 \end{eqnarray}
 Applying Case 1 to $B$, we know from (\ref{comp-count-2c}) that
 \begin{eqnarray}\label{c2b}
 c(B-U) \ge \sum_{uv \in E(J)} (c(B, uv)-2) + |U| .
 \end{eqnarray}
 Combining (\ref{c2a}), (\ref{c2c}) and (\ref{c2b}) gives
(\ref{comp-count}).

 \smallbreak\noindent
 \textbf{Case 3:} Suppose the vertices of $J$ do not all lie in a
single block of $G$.  Then there is a cutvertex $x$ of $G$ and vertices
$u, v$ of $J$ in distinct components of $G-x$.
 Let $G_1$ be the $x$-bridge in $G$ containing $u$, and $G_2$ the union
of all other $x$-bridges in $G$, which contains $v$.

  By Lemma \ref{N2Csubgraph}(c), $x$ is also a cutvertex of $J$
separating $u$ and $v$, and it follows from Lemma \ref{N2Csubgraph}(a)
that the vertex set of each $x$-bridge in $J$ lies in some $x$-bridge in
$G$.
 So for $i=1,2$  let $J_i$ be the union of the
$x$-bridges in $J$ whose vertex sets lie in $G_i$.  We have $u \in
V(J_1)$, $v \in V(J_2)$, $V(J_1) \cap V(J_2) = \{x\}$, and $J = J_1 \cup
J_2$.
 Since $|V(J_1)|, |V(J_2)| \ge 2$, we have $|E(J_1)|, |E(J_2)| > 0$, and
so $|E(J_1)|, |E(J_2)| < |E(J)|$.

 Let $U_i = V(J_i)$ for $i = 1,2$.
 Thinking of $G$, $G_1$ and $G_2$ as the unions of their blocks, it follows
from Observation \ref{n2cblock} that
 the N2Cs of $G_i$ are the N2Cs $\{u,v\}$ of $G$ with $\{u,v\} \subseteq
V(G_i)$, $J_i$ is a subgraph of the N2C graph of $G_i$,
 $c(G_i,uv)=c(G,uv)$ for $uv \in E(J_i)$, and $G_i$ is fundamental.
 Applying (\ref{comp-count}) to $G_i-U_i$ and $J_i$ by induction for
each $i$,
 \begin{eqnarray*}
 c(G-U) &=& \sum_{i=1}^2 c(G_i-U_i) \\
	&\ge& \sum_{i=1}^2 \left(
		\sum_{uv\in E(J_i)}(c(G_i,uv)-2)+
			\sum_{w\in U_i}(c(G_i-w)-1)+|U_i| \right) \\
	&=& \sum_{uv\in E(J)}(c(G,uv)-2)+
			\sum_{w\in U}(c(G-w)-1)-1+|U|+1 \\
	&&\quad\hbox{since $c(G_1-x)+c(G_2-x)=c(G-x)$
		and $|U_1|+|U_2| = |U|+1$} \\
  	&=& \sum_{uv\in E(J)}(c(G,uv)-2)+\sum_{w\in U}(c(G-w)-1)+|U|,
 \end{eqnarray*}
 proving (\ref{comp-count}).
 \end{proof}

 % Following theorem unchanged from submitted paper version (k170401)
 % except to add condition that graph is fundamental; OK.

 \begin{THM}\label{secondW}
 Let $\FF$ be the set of all N2Cs in a connected fundamental graph $G$, and let $W =
\bigcup \FF$ be the set of vertices used by $\FF$.
 For each $v \in V(G)$, let $\FF(v) = \{F \in \FF \;|\; v \in F\}$.
 If $f : V(G) \to \mZ$ satisfies
 \begin{equation*}
  c(G-S)\;\;\le\;\; \sum_{v\in S}(f(v)-1)
	\quad\hbox{for all $S \subseteq V(G)$ with $S \ne \emptyset$}
	\tag{\ref{TS}}
 \end{equation*}
 %  \begin{eqnarray}
 %   c(G-S)\le \sum_{v\in S}(f(v)-1)
 % 	\quad\hbox{for every $S \subseteq V(G)$ with $S \ne \emptyset$}
 % 	\label{compineq}
 %  \end{eqnarray}
 then there is a nonnegative integer function
$\omega$ on ordered pairs $(F,u)$ with $F\in \FF$ and $u\in F$
such that
 \begin{eqnarray}
 \omega(F, u)+\omega(F,v) &=& c(G,F)-2
	 \quad\hbox{for all $F = \{u,v\} \in \FF$, and}
	\label{wforf}\\
 \sum_{F\in \FF(u)}\omega(F,u) &\le&  f(u)-c(G-u)-1
	 \quad\hbox{for all $u \in V(G)$}.
	\label{wforu}
 \end{eqnarray}
 \end{THM}

 \begin{proof}
 If $v\in V(G)-W$, then $\FF(v) = \emptyset$ so $\sum_{F\in
\FF(v)}\omega(F,v)=0$.
 So (\ref{wforu}) is equivalent to $c(G-v)\le f(v)-1$
 which follows from (\ref{TS}) by taking $S=\{v\}$.
 Moreover, (\ref{wforf}) does not involve any vertices not in $W$.
Thus, in constructing $\omega$, the only vertices we need to be
concerned with are those in $W$.

 We associate with $G$ a network
$N$.
 Its vertex set is $V(N)=\{s,t\}\cup \FF\cup W$ where $s, t$ are new
vertices.
 Its arc set $A(N)$ consists of three subsets: $A_1 = \{sF \;|\; F \in
\FF\}$, $A_2 = \{Fu \;|\; F \in \FF, u \in W, u \in F\}$, and $A_3 =
\{ut \;|\; u \in W\}$.
 Each arc $a$ has a capacity $\cpc(a)$ defined as follows:
 $$
  \begin{array}{ll}
  \cpc(sF)=c(G,F)-2 & \hbox{if $sF\in A_1$,} \\
  \cpc(Fu)=\infty  & \hbox{if $Fu \in A_2$, \quad and}\\
  \cpc(ut)=f(u)-c(G-u)-1  & \hbox{if $ut \in A_3$.}
  \end{array}
 $$

 We claim  that a maximum $st$-flow $\flow$ of value $
  \Phi = \sum_{sF \in A_1} \cpc(sF)
 $
 in $N$ gives a desired way of distributing the weights on N2Cs to the
vertices, by taking $\omega(F, u) = \flow(Fu)$ for all $Fu \in A_2$.
 All arcs in $A_1$ must be saturated by such a flow, so flow conservation at
a vertex $F \in \FF$, where $F=\{u,v\}$, gives
 $$\omega(F, u) + \omega(F, v) = \flow(Fu) + \flow(Fv) = \flow(sF) =
\cpc(sF) = c(G, F)-2$$
 which verifies (\ref{wforf}), and flow conservation at a vertex $u \in
W$ gives
 $$\sum_{F \in \FF(u)} \omega(F, u) = \sum_{Fu {\fam0\ enters\ }u}
\flow(Fu) = \flow(ut) \le \cpc(ut) = f(u)-c(G-u)-1
 $$
 which verifies (\ref{wforu}).

 So assume that $N$ does not have a maximum flow of value $\Phi$; we
will show that this gives a contradiction.
 By the Max-Flow Min-Cut Theorem, $N$ has an $st$-cut
 $$ [S,T]=[\{s\}\cup \FF_1\cup W_1, \FF_2\cup W_2 \cup \{t\}]
	= [\{s\}, \FF_2] \cup [\FF_1, W_2] \cup [W_1, t]
 $$
 such that $\cpc(S,T)< \Phi$.
 Here $[Q,R]$ denotes all arcs from $Q$ to $R$, $\FF_1\cup \FF_2= \FF$,
$\FF_1 \cap \FF_2 = \emptyset$, $W_1\cup W_2=W$, and $W_1 \cap W_2 =
\emptyset$.

 If $\FF_2 = \FF$ then $\cpc(S,T) \ge \Phi$ which is a contradiction, so
$\FF_2 \ne \FF$ and $\FF_1 \ne \emptyset$.
 Since arcs in $A_2$ have infinite capacity and $\cpc(S,T) < \infty$, we
must have $[\FF_1, W_2] = \emptyset$. Therefore,
 $$ \cpc(S,T)=\sum_{F \in \FF_2} \cpc(sF) + \sum_{w \in W_1} \cpc(wt)
 = \sum_{\{u,v\}\in \FF_2} (c(G,uv)-2)+\sum_{w\in W_1}(f(w)-c(G-w)-1).
 $$
 Since $\cpc(S,T)<\Phi = \sum_{sF \in A_1} \cpc(sF) = \sum_{\{u,v\} \in
\FF} (c(G, uv)-2)$, we obtain
 $$
 \sum_{w\in W_1}(f(w)-c(G-w)-1)<\sum_{\{u,v\}\in \FF_1}(c(G,uv)-2),
 $$
or
 $$
 \sum_{\{u,v\}\in \FF_1}(c(G,uv)-2)+\sum_{w\in W_1}(c(G-w)-1)>\sum_{w\in W_1}(f(w)-2).
 $$

 Since $[\FF_1, W_2] = \emptyset$, $\bigcup \FF_1 \subseteq W_1$.
 So we may consider the graph $H_1$ with vertex set $W_1$ and edge set
$\FF_1$. By the Pigeonhole Principle, there is a component $J$ of $H_1$
with vertex set $U\subseteq W_1$ such that
 \begin{equation*}
 \sum_{uv\in E(J)}(c(G,uv)-2)+\sum_{w\in U}(c(G-w)-1)>\sum_{w\in U}(f(w)-2).
 \end{equation*}
 Combining this with Proposition \ref{com-counting}, we get
 \begin{eqnarray*}
   c(G-U)& \ge & \sum_{uv\in E(J)}(c(G,uv)-2)+\sum_{w\in U}(c(G-w)-1)+|U| \\
   & > &  \sum_{w\in U}(f(w)-2)+ |U| = \sum_{w\in U}(f(w)-1),
 \end{eqnarray*}
giving the required contradiction to (\ref{TS}).
 \end{proof}

\section{Proof of Theorem~\ref{k4mf-ftree}}
\label{proof-k4mf-ftree}

 In this section we use
Theorem~\ref{secondW} to prove Theorem~\ref{k4mf-ftree}.
 We start with some preliminary results.

 \begin{LEM}\label{N2C-subgraph-k4mf}
 Suppose $H$ is a subgraph of a $K_4$-minor-free graph $G$.
 Then every N2C of $H$ is also an N2C of $G$.
 \end{LEM}

 \begin{proof}
 Suppose $\{u,v\}$ is an N2C of $H$ that is not an N2C of $G$.  Then
$c(G, uv) < c(H, uv)$
 so there must be a path $P$ in $G-\{u,v\}$ between two components $C_1,
C_2$ of $H-\{u,v\}$ that attach at both $u$ and $v$, where we
may assume that the internal vertices of $P$ (if any) are not vertices
of $H$.  Let $C_3$ be a third component of $H-\{u,v\}$
attaching at both $u$ and $v$.  The subgraph
 $H[\{u,v\}
\cup V(C_1 \cup C_2 \cup C_3)] \cup P$ of $G$ contains a $K_4$ minor,
which is a contradiction.
 \end{proof}

 \begin{LEM}\label{k4-minor-free-no-N2C}
 If $G$ is $K_4$-minor-free and $G$ has no N2C, then $G$ is outerplanar.
 \end{LEM}

 \begin{proof} A graph is outerplanar if and only if it is
$K_4$-minor-free and $K_{2,3}$-minor-free. So assume that $G$ has a
$K_{2,3}$ minor.  Since $K_{2,3}$ has maximum degree $3$, the existence
of a $K_{2,3}$ minor implies that $G$ contains a subdivision $N$ of
$K_{2,3}$, consisting of two vertices $s_1$, $s_2$ of degree $3$ and
three internally disjoint $s_1s_2$-paths of length at least
$2$.
 But now $\{s_1, s_2\}$ is an N2C of $N$ and hence, by Lemma
\ref{N2C-subgraph-k4mf}, an N2C of $G$, which is a contradiction.
 \end{proof}

\begin{LEM}\label{root-edge}
 For every non-isolated vertex $x$ of a graph $G$, and for every block
$B$ of $G$ that contains $x$, there is $xy \in E(B)$ such that
$c(B, xy) = c(G,xy)\le 1$.
 \end{LEM}

 \begin{proof}
 For each $uv \in E(G)$, let $\AA(uv)$ be the set of
$\{u,v\}$-bridges of $G$ that attach at both $u$ and $v$.
 Suppose that $c(G, xy) \ge 2$ for every $xy \in E(B)$.
 Let $D_0 \in \AA(xy_0)$, $xy_0 \in E(B)$, attain the minimum
 $\min \{|V(D)| \;|\; D \in \AA(xy),\;  xy \in E(B)\}$.
 Let $xy_1$ be
an edge of $D_0$ incident with $x$.
 There is a path from $y_1$ to $y_0$ in $D_0-x$, so there is a cycle
containing $xy_0$ and $xy_1$, showing that $xy_1 \in E(B)$ also.
 Hence $c(G, xy_1) \ge 2$, and we can choose $D_1 \in \AA(xy_1)$ such
that $y_0 \notin V(D_1)$.  Now for each $z \in V(D_1) -
\{x, y_1\}$ there is a path in $D_1-x$ from $z$ to $y_1$; this path
avoids $y_0$ so it is also a path in $G-\{x,y_0\}$ from $z$ to $y_1 \in
V(D_0)$, showing that $z \in V(D_0)$ also.  Thus, $V(D_1) \subseteq
V(D_0)-\{y_0\}$, contradicting the minimality of $|V(D_0)|$.
 \end{proof}

\subsection{Trees in fundamental graphs}

We will first show the existence of an $f$-tree in a fundamental
$K_4$-minor-free graph, which will serve as a base case when
we find $f$-trees in general $K_4$-minor-free graphs by induction on
the number of vertices.

 The following fairly technical result is used to translate our weight
function from Theorem \ref{secondW} into spanning trees.
 Recall the definition of root edge, root block and root vertex from
Section \ref{sec-n2c}.
 Note that the special edges $\rr_0\ss_0, \rr_1\ss_1, \dots, \rr_k\ss_k$ below
always exist, by Lemma \ref{root-edge}.

 \begin{THM}\label{mainw2tree}
 Let $\FF$ be the set of all N2Cs in a connected fundamental
$K_4$-minor-free graph $G$ with $|V(G)| \ge 2$.
 For each $v \in V(G)$, let $\FF(v) = \{F \in \FF \;|\; v \in F\}$.
 Suppose there are $f : V(G) \to \mZ$ and a nonnegative integer function
$\omega$ on ordered pairs $(F,u)$ with $F\in \FF$ and $u\in F$ such that
 \begin{eqnarray}
 \omega(F, u)+\omega(F,v) &=& c(G,F)-2
	 \quad\hbox{for all $F = \{u,v\} \in \FF$, and}
	\label{wforf2}\\
 \sum_{F\in \FF(u)}\omega(F,u) + c(G-u) - 1 &\le&  f(u)-2
	 \quad\hbox{for all $u \in V(G)$.}
	\label{wforu2}
 \end{eqnarray}
 Choose a root edge $\rr_0\ss_0 \in E(G)$ such that $c(G, \rr_0\ss_0) \le 1$.
Let the blocks of $G$ be $B_0, B_1, B_2, \allowbreak\dots, B_k$, where
$\rr_0\ss_0\in E(B_0)$.
 For each $i = 1, 2, \dots, k$, let $\rr_i$ be the root vertex of $B_i$,
and choose $\rr_i\ss_i\in E(B_i)$ with $c(G,\rr_i\ss_i)\le 1$.

Then $G$ has a spanning tree $T$ such that $d_T(v)\le f(v)$ for all
$v\in V(G)$ and $d_T(v)\le f(v)-1$ for all $v\in \{\rr_0\} \cup \{\ss_0,
\ss_1, \dots, \ss_k\}$.
 Furthermore, for $0 \le i \le k$,
 \begin{equation}
   \rr_i\ss_i \in E(T) \quad\hbox{if $c(G,\rr_i\ss_i)=0$,}
 \qquad\hbox{and}\qquad
   \rr_i\ss_i \notin E(T) \quad\hbox{if $c(G,\rr_i\ss_i)=1$.}
 \label{rscond}
 \end{equation}
 \end{THM}

 \begin{proof}
 The proof is by induction on $|V(G)|$.  For the basis, if $|V(G)| = 2$
then $G=K_2$, $\FF$ is empty, conditions (\ref{wforf2}) and
(\ref{wforu2}) are trivially satisfied, $\rr_0\ss_0$ is the single edge, and
we can take $T=G$.  So we may assume that $|V(G)| \ge 3$.  There are two
cases.

 \smallskip\noindent
 \textbf{Case 1:} $G$ has no N2C.

 By Observation \ref{n2cblock} and Lemma \ref{k4-minor-free-no-N2C} each
block of $G$ is outerplanar and $c(G, \rr_i\ss_i) = c(B_i, \rr_i\ss_i) \le 1$
for $i = 0, 1, \dots, k$.  If $c(G, \rr_i\ss_i) = 0$ then $\rr_i \ss_i$ is a
cutedge of $G$, so $B_i = \rr_i\ss_i$, and we take $T_i=B_i$.  Otherwise,
$c(G, \rr_i\ss_i) = c(B_i, \rr_i \ss_i) = 1$, so $B_i$ contains a cycle using
$\rr_i \ss_i$.  Thus, $B_i$ is a $2$-connected outerplanar graph, which we
may embed in the plane with a hamiltonian cycle $C_i$ as its outer face.
 Since $c(B_i, \rr_i\ss_i) = 1$, $\rr_i\ss_i \in E(C_i)$, so we take $T_i =
C_i - \rr_i\ss_i$.  In either case, $T_i$ is a hamiltonian path and a
spanning tree in $B_i$, and hence $T = \bigcup_{i=0}^k T_i$ is a
spanning tree of $G$.

 % ORIGINAL
 Let us examine degrees in $T$.
 Since $G$ has no N2C, $\FF = \emptyset$, and inequality (\ref{wforu2})
just says that $c(G-v)-1\le f(v)-2$ for every $v\in V(G)$.
 Suppose first that $u \in V(G) - \{\rr_0, \ss_0, \rr_1, \ss_1,\dots, \rr_k,
\ss_k\}$.
 The cutvertices of $G$ are $\rr_1, \rr_2, \dots, \rr_k$, so $u$ is not a
cutvertex of $G$.  Hence $u$ lies in a single block, so by construction
of $T$, $d_T(u) = 2 \le f(u)$, as required.
 Suppose next that
 $u\in\{\rr_1,
\dots, \rr_k\}-(\{\rr_0\} \cup \{\ss_0, \ss_1, \dots, \ss_k\})$.
 By construction of $T$, $u$ has two incident edges of $T$ from its
parent block, and one incident edge from each child block, so $d_T(u) =
2 + (c(G-u)-1) \le f(u)$, as required.
 Finally, suppose that $u \in \{\rr_0\} \cup \{\ss_0, \ss_1, \dots, \ss_k\}$.
Then $u$ has one incident edge of $T$ from each block to which it
belongs, so $d_T(u)=c(G-u)\le f(u)-1$, as required.

 % ALTERNATIVE VERSION
 %  Let us examine degrees in $T$.
 %  Since $G$ has no N2C, $\FF = \emptyset$, and inequality (\ref{wforu2})
 % just says that $c(G-v)-1\le f(v)-2$ for every $v\in V(G)$.
 %  Given a vertex $v$, let $B_i$ be the block containing $v$ that is
 % closest to the root block $B_0$.  This is either the parent of $v$ in
 % the block-cutvertex tree if $v$ is a cutvertex, or the unique block
 % containing $v$ otherwise.  If $v \in \{\rr_0\} \cup \{\ss_0, \ss_1,
 % \dots, \ss_k\}$ then $v$ has degree $1$ in $T_i$, and otherwise $v$ has
 % degree $2$ in $T_i$.  If $v$ also belongs to some other block $B_j$, $j
 % \ne i$, then $v=\rr_j$ and so $v$ has degree $1$ in $T_j$.  Thus,
 % $d_T(v) = c(G-v) \le  f(v)-1$ if $v \in \{\rr_0\} \cup \{\ss_0, \ss_1,
 % \dots, \ss_k\}$, and $d_T(v) = c(G-v)+1 \le f(v)$ otherwise.

Furthermore, by construction of $T$, for $0 \le i \le k$ we have
$\rr_i\ss_i\in E(T)$ if $c(G,\rr_i\ss_i)=0$ and $\rr_i\ss_i\not\in E(T)$ if
$c(G,\rr_i\ss_i)=1$.

 \smallskip\noindent
 \textbf{Case 2:} $G$ contains an N2C.

 Let $\xy$ be an N2C of $G$, with $\xy$ contained in a block $B_m$.
Since $c(G, xy) \ge 3$, we can choose a nontrivial $\xy$-bridge $D$ that
attaches at both $x$ and $y$ so that $\rr_0 \ss_0, \rr_m \ss_m \notin
E(D)$.
 Let $G_1=G-(V(D)-\xy)$, so $G_1$ is the union of all $\xy$-bridges
other than $D$, including all bridges that attach only at $x$ or only at
$y$.  We may assume that the blocks of $G_1$ are $B_0, B_1, \dots,
B_{m-1}$ and the new block $B_m'$ consisting of all $\xy$-bridges of
$B_m$ other than $D \cap B_m$.
 Let $G_2 = D \cup xay$ where $a$ is a new
vertex; $G_2$ is a minor of $G$, so it is $K_4$-minor-free.
 Neither $x$ nor $y$ is a cutvertex of $G_2$, and $\xy$ is not an
N2C of $G_2$ since $c(G_2, xy)=2$.
 The blocks of $G_2$ are all the blocks of $G$ included in $D$, which
are $B_{m+1}, B_{m+2}, \dots, B_k$, and the new block $B_m'' = (D \cap B_m) \cup xay$.  No edge of
$B_{m+1}, B_{m+2}, \dots, B_k$ is incident with $x$ or $y$.

 We can apply Observations \ref{n2cblock} and \ref{n2cbridge}, breaking
$G_1$, $G_2$ and $G$ up into blocks and into $\xy$-bridges in $B_m'$
and $B_m''$, which are both subgraphs of
 $B_m \cup xay$.  It follows that
$G_1$ and $G_2$ are both fundamental.  Also, if $\FF_j$, $j = 1$ or $2$,
is the set of N2Cs of $G_j$, then $\FF_1$ and $\FF_2$ are disjoint and $\FF
= \FF_1 \cup \FF_2 \cup \{\xy\}$.  Possibly $\xy \in \FF_1$,
but $\xy \notin \FF_2$.  If $\uv \in \FF_j$ and $\uv \ne \xy$ then
$c(G,uv) = c(G_j, uv)$.

 As $\xy$ is an N2C, $\omega(\xy,x)\ge1$ or
$\omega(\xy,y)\ge1$; without loss of generality, assume that
$\omega(\xy,x)\ge1$.

 For $F \in \FF_1 \cup \{\xy\}$ and $u
\in F$, recall that $\de_x(x)=1$ and $\de_x(y)=0$, and define
 $$\omega_1(F, u) = \begin{cases}
	\omega(\xy, u)-\de_x(u) & \hbox{if $F=\xy$,}\\
	\omega(F, u) & \hbox{otherwise.}
  \end{cases}
 $$
 Then $\omega_1(F, u) \ge 0$ whenever it is defined.
 For $F \in \FF_2$ and $u \in F$ define $\omega_2(F, u) = \omega(F, u)$.

 For $u \in V(G_1)$ define
 $$
 f_1(u) = \begin{cases}
	f(u)-\de_x(u)-\ffts{u} & \hbox{if $u \in \xy$,}\\
	f(u)  & \hbox{otherwise.}
 \end{cases}
 $$

 For $u \in V(G_2)$ define
 $$
 %\qquad\hbox{and}\qquad
 f_2(u) = \begin{cases}
	\ffts{u} + 2 & \hbox{if $u \in \xy$,}\\
	2 & \hbox{if $u = a$,} \\
	f(u)  & \hbox{otherwise.}
 \end{cases}
 $$

 \begin{CLAu}
 For $j=1$ and $2$, (\ref{wforf2}) and (\ref{wforu2}) hold for $f_j$ and
$\omega_j$ in $G_j$.
 \end{CLAu}

 \begin{claimproof}
 If $F \in \FF_j$ and $F \ne \xy$, then (\ref{wforf2}) for $F$ in
$G_j$ is just a rewriting of (\ref{wforf2}) for $F$ in $G$.
 If $u \in V(G_j)$ and $u \notin \{x, y, a\}$ then
(\ref{wforu2}) for $u$ in $G_j$ is just a rewriting of (\ref{wforu2})
for $u$ in $G$.

 If $F=\xy \in \FF_1$ then (\ref{wforf2}) holds for $\xy$ in $G_1$
because
 $$\omega_1(F,x)+\omega_1(F,y) = \omega(F,x)+\omega(F,y)-1 \le
 c(G,F)-2-1 = c(G_1, F)-2 .$$

 Suppose $u \in \xy$.  If $\xy \in \FF_1$ then $\FF(u) =
\FF_1(u) \cup \FF_2(u)$ and (\ref{wforu2}) for $u$ in $G$ can be
rewritten as
 $$\fft11u + \de_x(u) + \ffts u + c(G_1-u)-1 \le f_1(u) + \de_x(u) +
\ffts u - 2$$
 which gives (\ref{wforu2}) for $u$ in $G_1$ after cancelling
$\de_x(u) + \ffts u$ on both sides.
 If $\xy \notin \FF_1$ then $\xy$ is included in $\FF(u)$ but
not in $\FF_1(u)$, so this inequality is modified by adding a
nonnegative term $\omega_1(\xy,u)$ on the left, which can be deleted, so
(\ref{wforu2}) still holds for $u$ in $G_1$.
 Also, (\ref{wforu2}) holds (with equality) for $u$ in $G_2$ by
definition of $f_2(u)$ and because $\omega_2$ is just a restriction of
$\omega$ and $c(G_2,u)=1$.

 For $u = a$, $\FF_2(a) = \emptyset$, $c(G_2-a)=1$, and $f_2(a)=2$,
so (\ref{wforu2}) holds for $a$ in $G_2$.
 \end{claimproof}

 In $G_1$ we can still choose $\rr_0 \ss_0$ as the root edge, and $\rr_1
\ss_1, \rr_2 \ss_2, \dots \rr_m \ss_m$ as the other special edges (where
$\rr_m \ss_m$ is now in $B_m'$ instead of $B_m$).  By the claim we can
apply induction to find a spanning tree $T_1$.
 In $G_2$ we choose $ax \in E(B_m'')$ as the root edge, and we can still
choose $\rr_{m+1} \ss_{m+1}, \dots, \rr_k \ss_k$ as the other special
edges.  By the claim we can apply induction to find a spanning tree
$T_2$.  Since $c(G_2,ax)=1$, we have $ax \notin E(T_2)$ by
(\ref{rscond}), so $ay \in E(T_2)$.  Thus, $T_2' = T_2-ay$ is a spanning
tree of $D = G_2-a$.

 Now $T_1$ and $T_2'$ both contain $xy$-paths.
 So $T_1 \cup T_2'$ is a spanning subgraph of $G$ containing a single
cycle $C$ through $x$ and $y$.  Let $yz$ be the edge of $C \cap
G_2$ incident with $y$.  Then $T = (T_1 \cup T_2') - yz$ is a spanning
tree of $G$.

 For all $v \in V(G)$, define $\sigma(v)$ to be $1$ if $v \in \{r_0\}
\cup \{s_0, s_1, \dots, s_k\}$, and $0$ otherwise.  If $u \in V(G) -
\xy$, then $u \in V(G_j)-\xy$ for $j = 1$ or $2$, and $d_T(u) \le
d_{T_j}(u) \le f_j(u) - \sigma(u) = f(u) - \sigma(u)$ as required.

 We also want to show that $d_T(u) \le f(u) - \sigma(u)$ when $u \in
\xy$.
 We know that no edge of $B_{m+1}, B_{m+2}, \dots, B_k$ is incident with
$u$, so if $\sigma(u)=1$, it is because $u$ is incident with
an edge $\rr_i \ss_i$ for $0 \le i \le m$; this edge belongs to $G_1$.
 Thus, $d_{T_1}(u) \le f_1(u)-\sigma(u)$ by construction of $T_1$.
 Since $ax$ is the root edge of $G_2$, $d_{T_2'}(x) =
d_{T_2}(x) \le f_2(x)-1$, so
 $$d_T(x) = d_{T_1}(x) + d_{T_2'}(x) \le (f_1(x)-\sigma(x))+(f_2(x)-1)
	 = f(x)+1-\de_x(x)-\sigma(x) = f(x)-\sigma(x).$$
 Since we delete two edges $ay$ and $yz$ from $T_1 \cup T_2$
when forming $T$,
 $$d_T(y) = d_{T_1}(y) + d_{T_2}(y)-2 \le (f_1(y)-\sigma(y)) +
	f_2(y) - 2 =
	f(y)-\de_x(y)-\sigma(y) = f(y)-\sigma(y).$$

 Finally we verify (\ref{rscond}).
 For each $i$, $0 \le i \le k$, $\rr_i \ss_i \in E(G_j)$ for a unique $j
\in \{1,2\}$.
 If $c(G, \rr_i\ss_i)=0$ then $\rr_i\ss_i$ is a cutedge of $G$ and so
must belong to the spanning tree $T$.
 Otherwise, $c(G,\rr_i\ss_i) = c(G_j,\rr_i\ss_i)=1$ and so $\rr_i\ss_i
\notin E(T_j)$, and hence $\rr_i \ss_i \notin E(T)$.
 \end{proof}

 Combining Theorems \ref{secondW} and \ref{mainw2tree} we obtain our
result on spanning trees in fundamental $K_4$-minor-free graphs.  The
hypotheses (\ref{wforf2}) and (\ref{wforu2}) of Theorem \ref{mainw2tree}
are just slightly rewritten versions of the conclusions (\ref{wforf})
and (\ref{wforu}) of Theorem \ref{secondW}.
 Note that Theorem \ref{mainw2tree} does not handle the case where
$|V(G)|=1$, but that is trivial.

 \begin{THM}\label{maintreefund}
 Let $G$ be a fundamental connected $K_4$-minor-free graph.
 Suppose that $f : V(G) \to \mZ$ satisfies

 \begin{equation*}
  c(G-S)\;\;\le\;\; \sum_{v\in S}(f(v)-1)
	\quad\hbox{for all $S \subseteq V(G)$ with $S \ne \emptyset$.}
	\tag{\ref{TS}}
 \end{equation*}
 %  \begin{eqnarray*}
 %   c(G-S)\le \sum_{v\in S}(f(v)-1)
 % 	\quad\hbox{for every $S \subseteq V(G)$ with $S \ne \emptyset$.}
 % 	%\label{compineq}
 %  \end{eqnarray*}
 Choose a root edge $\rr_0\ss_0 \in E(G)$ such that $c(G, \rr_0\ss_0) \le 1$.
Let the blocks of $G$ be $B_0, B_1, B_2, \allowbreak\dots, B_k$, where
$\rr_0\ss_0\in E(B_0)$.
 For each $i = 1, 2, \dots, k$, let $\rr_i$ be the root vertex of $B_i$,
and choose $\rr_i\ss_i\in E(B_i)$ with $c(G,\rr_i\ss_i)\le 1$.

Then $G$ has a spanning tree $T$ such that $d_T(v)\le f(v)$ for all
$v\in V(G)$ and $d_T(v)\le f(v)-1$ for all $v\in \{\rr_0\} \cup \{\ss_0,
\ss_1, \dots, \ss_k\}$.
 Furthermore, for $0 \le i \le k$,
   $\rr_i\ss_i \in E(T)$ if $c(G,\rr_i\ss_i)=0$,
  and
   $\rr_i\ss_i \notin E(T)$ if $c(G,\rr_i\ss_i)=1$.
 \end{THM}

This theorem is quite technical, and only applies to fundamental
$K_4$-minor-free graphs, but it gives strong results in that situation,
particularly for graphs with many cutedges.
 For example, neither the general result Theorem \ref{ENV-ftree} nor our
main result Theorem \ref{k4mf-ftree} can prove that when $G$ is a tree,
$G$ has a spanning $f$-tree where $f=d_G$ is just the degree function in
$G$.
 However, we can get this result from Theorem \ref{maintreefund}
with $f(v) = d_G(v)+1$ for all $v$, $\rr_0 \ss_0$ an arbitrary
edge, and $\{\ss_1, \ss_2, \dots, \ss_k\} = V(G) - \{\rr _0,
\ss_0\}$.

\subsection{Trees
in general graphs}
	
	In this subsection, we find spanning $f$-trees in general
$K_4$-minor-free graphs.  We will use the Component Condition
(\ref{TS}), as stated in Theorem \ref{k4mf-ftree}, often.

	\begin{LEM}\label{LEM:2-cut-no-uv}
 Let $G$ be a 2-connected graph and $f : V(G) \to \mZ$.
 If  $G$ and $f$ satisfy (\ref{TS}),
then for every 2-cut $\{u,v\}$ of $G$ with $uv\in E(G)$,
$G-uv$ and $f$ also satisfy
(\ref{TS}).
	\end{LEM}
	
	\begin{proof}

 Let $G'=G-uv$, and let $S\subseteq V(G')$ with $S \ne \emptyset$.
 If $\uv \cap S \ne \emptyset$ then $G'-S = G-S$, and if $u$ and $v$ are
in the same component of $G'-S$ then $c(G'-S)=c((G'-S) \cup uv) = c(G-S)$; in either case
(\ref{TS}) holds for $G'$, $f$ and $S$.
 So we may assume that $u, v \notin S$ and $u$ and $v$ are in distinct
components of $G'-S$.

 Let $D_1, D_2, \dots, D_t$ be the $\uv$-bridges in $G'$, and let $S_i =
S \cap V(D_i)$ for each $i$.  Then $S = \bigcup_{i=1}^t S_i$. As $\uv$
is a $2$-cut, $t \ge 2$.  Now $S_i$ must separate $u$ and $v$ in $D_i$,
so the set of components of $D_i-S_i$ can be written as $\AA_i \cup
\{C_i^u, C_i^v\}$ where $u \in V(C_i^u)$, $v \in V(C_i^v)$, and $\AA_i$
contains all other components.  The set of components of $G-S_i$ is then
$\AA_i \cup  \{C_i^u \cup C_i^v \cup \bigcup_{j \ne i} D_j\}$, so
$c(G-S_i) = |\AA_i|+1$.  The set of components of $G'-S$ is
 $\bigcup_{i=1}^t \AA_i \cup
 \{\bigcup_{i=1}^t C_i^u, \bigcup_{i=1}^t C_i^v\}$.
   Since $t \ge 2$ and (\ref{TS}) holds for
$G$ and $f$, we have
 \begin{eqnarray*}
  c(G'-S) &=& \sum_{i=1}^t |\AA_i| + 2 \;=\;
	\sum_{i=1}^t (c(G-S_i)-1) + 2 \\
  &\le& \sum_{i=1}^t \sum_{v \in S_i} (f(v)-1) - t + 2
	= \sum_{v \in	S} (f(v)-1) -t + 2
	\le \sum_{v \in	S} (f(v)-1)
 \end{eqnarray*}
 and so (\ref{TS}) holds for $G'$, $f$ and $S$.
 \end{proof}
	
 If every $2$-cut $\{u,v\}$ of a graph $G$ satisfies $uv \notin E(G)$,
we say that $G$ is \emph{$2$-cut-reduced}.
 If $G$ is $2$-connected and $\uv$ is a $2$-cut of $G$ with $uv \in
E(G)$, then $G-uv$ is still $2$-connected.  Therefore, we can apply
Lemma \ref{LEM:2-cut-no-uv} repeatedly.  We may create new $2$-cuts when
we do, but we can only create a bounded number, so we obtain the
following.

 \begin{COR}\label{all-2-cut-no-uv}
 Let $G$ be a 2-connected graph and $f : V(G) \to \mZ$.
 Assume that $G$ and $f$ satisfy (\ref{TS}).  Then $G$ has
a $2$-cut-reduced spanning subgraph $G'$ such that $G'$ and $f$ also
satisfy (\ref{TS}).
 \end{COR}

 We write a $uv$-path $P$ as $P[u,v]$ to emphasize its endvertices.
 Suppose the block-cutvertex tree of a graph $H$ is a path.
 Then we write $H=v_0B_1v_1B_2v_2\dots v_{t-1}B_tv_t$, where
 each $B_i$ is a block of $H$,
 $v_0 \in V(B_1) -\{v_1\}$, $v_t \in V(B_t) - \{v_{t-1}\}$,
 and each $v_i$, $i=1,\dots, t-1$, is a cutvertex of $H$ with $v_i\in
V(B_i)\cap V(B_{i+1})$.
 (If $H$ has only one vertex we take $t=0$ so that $H = v_0$.)
 We say $H$ is a \emph{chain of blocks from $v_0$ to $v_t$}. If $t \ge
2$ we say the chain of blocks is \emph{nontrivial}.

 \begin{LEM}\label{edge-deletion}
 Let $G$ be a $2$-connected graph and $xy \in E(G)$.

 \noindent
 \textnormal{(a)} Then $G-xy$ is a chain of blocks from $x$ to $y$.

 \noindent
 \textnormal{(b)} Moreover, if $G$ is a $2$-cut-reduced $K_4$-minor-free graph, then
$G-xy$ is a nontrivial chain of blocks.
	\end{LEM}
	
	\begin{proof}
 (a) Since $G$ is $2$-connected, if $G-xy$ has a cutvertex
then every leaf block of $G-xy$ contains $x$ or $y$ as a non-cutvertex
vertex.
 Hence, $G-xy$ is either a single block or has only two leaf blocks, and
is a chain of blocks from $x$ to $y$.

 (b) If $G-xy$ is 2-connected, then there are internally disjoint $xy$-paths
$P_1$ and $P_2$ in $G-xy$.
 Since $G$ is $2$-cut-reduced, $\{x,y\}$ is not a 2-cut of $G$, so there
is a path $P$ in $G-\{x,y\}$ connecting $P_1$ to $P_2$, where we may
assume that $|V(P)\cap V(P_i)|=1$, $i=1,2$.
 Then $P_1\cup P_2\cup P\cup xy$
 contains a $K_4$ minor, which is a
contradiction.  Hence $G-xy$ is not $2$-connected, so it is a nontrivial
chain of blocks.
 \end{proof}

 If $H$ has at least two blocks, $B$ is a leaf block of
$H$, and the cutvertex of $H$ in $B$ is $v$, then $H \ominus
B$ denotes $H - (V(B) - \{v\})$.

 \begin{LEM}\label{fund-subgraph}
 Let $G$ be a $2$-cut-reduced $2$-connected $K_4$-minor-free graph.
 If $G$ is not fundamental and $u\in V(G)$, then there exist $xy\in
E(G)$ and a leaf block $B$ of $G-xy$ such that each of $x$ and $y$ is contained in an N2C of $G$, $B$ is
fundamental, and $u \in V((G-xy) \ominus B)$.
 \end{LEM}
	
 \begin{proof}
 Since $G$ is not fundamental there is at least one edge $xy \in E(G)$
such that both $x$ and $y$ are contained in N2Cs of $G$.
 By Lemma \ref{edge-deletion}, $G-xy$ is a nontrivial chain of blocks
from $x$ to $y$.  For such an $xy$ there is at least one leaf
block $B$ of $G-xy$ so that $u \in V((G-xy) \ominus B)$.  Choose such an
$xy$ and such a $B$ so that $|V((G-xy) \ominus B)|$ is as large as
possible.
 We may assume that $G-xy = v_0 B_1 v_1 \dots v_{t-1} B_t v_t$ where
$x=v_0$, $y = v_t$ and $B=B_1$.

 We claim that $B_1$ is fundamental.  Suppose not.
 Then there is $x_1 y_1 \in E(B_1)$ with both $x_1$ and $y_1$ in N2Cs of
$B_1$.
 By Lemma \ref{N2C-subgraph-k4mf}, both $x_1$ and $y_1$ are also in N2Cs
of $G$.
 Since $B_1$ has N2Cs it has at least five vertices and is
$2$-connected.
 By Lemma \ref{edge-deletion}(a), $B_1-x_1y_1$ is a chain of blocks
$w_0 D_1 w_1 \dots w_{s-1} D_s w_s$ (possibly $s=1$, but $s \ne 0$).
 Choose $h$ and $k$ such that $h \le k$, $x$ and $v_1$
are vertices of $w_h D_{h+1} w_{h+1} \dots w_{k-1} D_k w_k$, and $k-h$
is as small as possible.  Since $x \ne v_1$ we have $h < k$.
 We may assume that $x_1 = w_0$, $y_1 = w_s$, $x \in D_{h+1} -
\{w_{h+1}\}$ and $v_1 \in D_k - \{w_{k-1}\}$.

By Lemma \ref{edge-deletion}, $G-x_1 y_1$ is a nontrivial chain of
blocks, which must be
 $w_0 D_1 \dots \allowbreak w_{h-1} D_h w_h
D^* w_k D_{k+1} \dots \allowbreak D_s w_s$
 where $D^*$ is a block
  $(x D_{h+1} w_{h+1} D_{h+2} \dots\allowbreak w_{k-1} D_k v_1 B_2
v_2 \dots\allowbreak v_{t-1} B_t y) \allowbreak \cup \allowbreak xy$.
 Since $G-x_1y_1$ is nontrivial, either $h > 0$ or $k < s$ (hence $s
\ge 2$).  If $h > 0$ then $D_1$ is a leaf block of $G-x_1 y_1$, $u \in
V((G-xy) \ominus B_1) = V(B_2 \cup \dots \cup B_t) \subset V(D^*)
\subseteq V((G-x_1y_1) \ominus D_1)$ and $|V((G-x_1y_1) \ominus D_1)|
\ge |V(D^*)| > |V((G-xy) \ominus B_1)|$.  Thus, $x_1 y_1$ and $D_1$
contradict the choice of $xy$ and $B=B_1$.  We obtain a similar
contradiction from $x_1 y_1$ and $D_s$ if $k < s$.

 Hence, $B=B_1$ is fundamental, as claimed, which proves the lemma.
 \end{proof}

We are now ready to prove Theorem~\ref{k4mf-ftree}, which we restate for
convenience.

 \begin{THM-k4mf-ftree}
 \mainhypothesis
 \begin{equation*}
 c(G-S) \;\;\le\;\; \sum_{v\in S}(f(v)-1)
	\quad\hbox{for all $S\subseteq V(G)$ with $S\ne \emptyset$.}
	\tag{\ref{TS}}
 \end{equation*}
 \mainconclusion
 \end{THM-k4mf-ftree}

\begin{proof}
 The proof is by induction on $|V(G)|$.
 The conclusion is true by Theorem
\ref{maintreefund} if $|V(G)| \le 4$, since $G$ has no N2Cs and is
fundamental. So we assume that $|V(G)|\ge 5$.
 We consider two cases.

 \smallskip\noindent
\textbf{Case 1:} $G$ has a cutvertex $x$.

 Let $G_1$ be one $x$-bridge, and let $G_2$
be the union of the remaining $x$-bridges.
 If $\zz \in V(G_1)-\{x\}$ let $\zz_1=\zz$ and $\zz_2=x$, if $\zz \in
V(G_2)-\{x\}$ let $\zz_1=x$ and $\zz_2=\zz$, and if $\zz=x$ let
$\zz_1=\zz_2=x$.
 For every $v \in V(G_i)-\{x\}$, $i=1,2$, let $f_i(v)=f(v)$.
 Then each $G_i$ and $f_i$, $i = 1$ or $2$, satisfy (\ref{TS})
for all $S$ that do not contain $x$.

 Let $f_2(x)$ be the minimum integer such that $G_2$ and $f_2$ satisfy
(\ref{TS}) for all $S$ containing $x$, and let $f_1(x)=f(x)-f_2(x)+1$.
 We claim that $G_1$ and $f_1$ also satisfy (\ref{TS}) for sets $S$
containing $x$.  If not, there is $U \subseteq V(G_1)$ with $x \in U$ so
that
 $$c(G_1-U) \ge \sum_{v\in U}(f_1(v)-1) + 1
	= \sum_{v \in U-\xx} (f(v)-1) + f_1(x). $$
 By minimality of $f_2(x)$ there is $W \subseteq V(G_2)$ with $x \in W$
so that
 $$c(G_2-W) = \sum_{v \in W} (f_2(v)-1)
	= \sum_{v \in W-\xx} (f(v)-1) + f_2(x)-1. $$
 Thus, if $S = U \cup W$ we have
 \begin{eqnarray*}
 c(G-S) &=& c(G_1-U) + c(G_2-W)
	\;=\; \sum_{v \in S-\xx} (f(v)-1) + f_1(x)+f_2(x)-1 \\
	&=& \sum_{v \in S-\xx} (f(v)-1) + f(x)
	\;>\; \sum_{v \in S} (f(v)-1), \\
 \end{eqnarray*}
 contradicting (\ref{TS}) for $G$ and $f$.

 Thus, for each $i=1,2$, $G_i$ and $f_i$ satisfy (\ref{TS}), so by
induction there is a spanning $f_i$-tree $T_i$ of $G_i$ with
$d_{T_i}(\zz_i) \le f_i(\zz_i)-1$.  Let $T = T_1 \cup T_2$, which is a
spanning tree of $G$.
 If $v \notin \{\zz,x\}$ then $v \in V(G_i)-\xx$ for some $i$, and
$d_T(v) = d_{T_i}(v) \le f_i(v) = f(v)$.
 If $z \in V(G_i)-\xx$ then $d_T(\zz) = d_{T_i}(\zz) \le f_i(\zz)-1 =
f(\zz)-1$ since $\zz=\zz_i$, and $d_T(x) = d_{T_1}(x)+d_{T_2}(x)\le
f_1(x)+f_2(x)-1 = f(x)$ since $x=\zz_{3-i}$.  If $\zz = x$ then $d_T(x)
\le (f_1(x)-1)+(f_2(x)-1) = f(x)-1$ since $x = \zz_1 = \zz_2$.  Thus,
$T$ is a spanning $f$-tree with $d_T(\zz) \le f(\zz)-1$.

 \smallskip\noindent \textbf{Case 2:} $G$ has no cutvertex; since
$|V(G)| \ge 5$, $G$ is $2$-connected.

 We may assume that $G$ is not fundamental, by Theorem
\ref{maintreefund}, and that $G$ is $2$-cut-reduced, by Corollary
\ref{all-2-cut-no-uv}. By Lemmas \ref{edge-deletion} and
\ref{fund-subgraph}, $G$ contains an edge $xy$ with both $x$ and $y$ in
N2Cs of $G$ such that some leaf block $B$ of $G-xy$, which is a
nontrivial chain of blocks from $x$ to $y$, is fundamental and $\zz\in
V((G-xy)\ominus B)$. Without loss of generality $x \in V(B)$ and so
$G-xy=xB_1v_1B_2v_2\dots B_ty$ where $t \ge 2$ and $B=B_1$.  Write $L_2
= (G - xy) \ominus B_1 = v_1 B_2 v_2 \dots v_{t-1} B_t v_t$, so
that $z \in V(L_2)$.
 Since $x$ and $y$ are in N2Cs, $d_G(x), d_G(y) \ge 3$ and hence
$|V(B_1)|, |V(L_2)| \ge 3$.

 Let $G_1 = B_1 \cup xa_1v_1$ where
$a_1$ is a new vertex, and let $G_2 = L_2
\cup v_1y$.  Both $G_1$ and $G_2$ are minors of $G$ and hence
$K_4$-minor-free.
 For $i = 1, 2$ define $f_i(v) = f(v)$ for $v \in V(G_i) - \{v_1,
a_1\}$.
 Then for $S \subseteq V(G_2)$ with $v_1 \notin S$ we have $c(G-S) =
c(G_2-S)$ and $f_2(v) = f(v)$ for all $v \in S$, so (\ref{TS}) holds for
$G_2$, $f_2$ and $S$.
 Choose $f_2(v_1)$ to be the minimum integer such that (\ref{TS}) holds
for $G_2$ and $f_2$, then by minimality there is $W \subseteq V(G_2)$
with $v_1 \in W$ so that
 $$c(G_2-W) = \sum_{v \in W} (f_2(v)-1)
	= \sum_{v \in W-\vv} (f(v)-1) + f_2(v_1)-1. $$
 Let $f_1(v_1) = f(v_1)-f_2(v_1)+2$ and $f_1(a_1) = |V(G_1)|+1$ (note
that our theorems do not require that $f(v) \le d_G(v)$).
 We claim that (\ref{TS}) holds for $G_1$ and $f_1$.  This is true for
any $S \subseteq V(G_1)$ with $a_1 \in S_1$, by choice of $f_1(a_1)$.
 If we have nonempty
 $S \subseteq V(G_1)$
 with $a_1, v_1 \notin S$,
 $$c(G_1-S) = c(G-S) \le \sum_{v \in S} (f(v)-1) = \sum_{v \in S}
(f_1(v)-1)$$
 and so (\ref{TS}) holds in this situation as well.  So if (\ref{TS})
does not hold for $G_1$ and $f_1$ there is $U \subseteq V(G_1)$ with
$a_1 \notin U$, $v_1 \in U$ and
 $$c(G_1-U) \ge \sum_{v \in U} (f_1(v)-1) + 1= \sum_{v \in U-\vv} (f(v)-1) +
f_1(v_1).$$
 Let $S = U \cup W$ then $a_1 \notin S$, $v_1 \in S$.
 % NEW
 Note that $c(G-xy-S) = c(B_1-U)+c(L_2-W)$.  If $x \notin U$ then
$c(G-S) \ge c(G-xy-S)-1$ and $c(G_1-U) = c(B_1-U)$, so
 $$c(G-S) \ge c(B_1-U) + c(L_2-W) - 1 = c(G_1-U) + c(G_2-W) - 1$$
and if $x \in U$ then
 $c(G-S)=c(G-xy-S)$ and $c(G_1-U)=c(B_1-U)+1$, so
 $$c(G-S) = c(B_1-U) + c(L_2-W) = c(G_1-U) + c(G_2-W) - 1.$$
 Thus, in either case,
 \begin{eqnarray*}
  c(G-S) &\ge& c(G_1-U) + c(G_2-W) - 1 \\
	&\ge& \sum_{v \in S-\vv} (f(v) - 1) + f_1(v_1) + f_2(v_1)-1 -1 \\
	&=& \sum_{v \in S-\vv} (f(v) - 1) + f(v_1)
	\;>\; \sum_{v \in S} (f(v)-1) \\
 \end{eqnarray*}
 contradicting (\ref{TS}) for $G$ and $f$.  Thus, $G_1$ and $f_1$
satisfy (\ref{TS}).

 Now $a_1$ does not belong to any N2C of $G_1$ since it has degree $2$,
so $G_1$ is fundamental.  Choosing root edge $r_0s_0 = a_1x$ and
applying Theorem \ref{maintreefund}, there is a spanning $f_1$-tree
$T_1$ of $G_1$ with
 $d_{T_1}(x) \le f_1(x)-1$
 and $a_1 x \notin E(T_1)$.
 Thus, $a_1 v_1 \in E(T_1)$ and $T_1' = T_1 - a_1$ is a spanning
$f_1$-tree of $G_1-a_1 = B_1$ with $d_{T_1'} (v) \le f_1(v)-1$ for $v
\in \{x,v_1\}$.

 Since $|V(B_1)| \ge 3$, $|V(G_2)| < |V(G)|$, so by induction there is a
spanning $(f_2-\de_\zz)$-tree $T_2$ of $G_2$ (i.e., an $f_2$-tree with
$d_{T_2}(\zz) \le f_2(\zz)-1$).  Suppose that some such tree $T_2$ has
$v_1 y \in E(T_2)$.  Then
 $T = T_1' \cup (T_2-v_1y) \cup xy$
 has $d_T(v) \le f_1(v) = f(v)$ for $v \in V(B_1)-\{x,v_1\}$, $d_T(v)
\le f_2(v) - \de_\zz(v) = f(v)-\de_\zz(v)$ for $v \in V(L_2)-\{v_1,y\}$,
and
 \begin{eqnarray*}
 d_T(x) &\le& (f_1(x)-1) + 1 = f(x), \\
 d_T(v_1) &\le& (f_1(v_1)-1) + (f_2(v_1)-\de_\zz(v_1)-1) =
	f(v_1)-\de_\zz(v_1), \\
 d_T(y) &\le& f_2(y)-\de_\zz(y)-1+1 = f(y)-\de_\zz(y).
 \end{eqnarray*}
 Thus, $T$ is our desired spanning tree.  We may henceforth
assume that every $(f_2-\de_\zz)$-tree $T_2$ of $G_2$ has $v_1 y \notin
E(T_2)$, and so is a spanning tree of $L_2$.

 Now define $g_1$ on $V(G_1)$ by
 $g_1(x) =
f_1(x)+1=f(x)+1$, $g_1(v_1)=f_1(v_1)-1=f(v_1)-f_2(v_1)+1$,
and $g_1(v) = f_1(v)$ for $v \in V(G_1) -
\{x,v_1\}$.  Suppose that $G_1$ and $g_1$ satisfy (\ref{TS}).  By
similar reasoning to that for $T_1'$ above, there is a spanning
$g_1$-tree $R_1'$ of $B_1$ with $d_{R_1'}(v) \le g_1(v)-1$ for $v \in
\{x, v_1\}$.  Taking $T = R_1' \cup T_2$ (remembering that $T_2$ is now
a spanning tree of $L_2$) we get a spanning tree of $G$ with $d_T(v) \le
g_1(v) = f(v)$ for $v \in V(B_1)-\{x,v_1\}$, $d_T(v) \le f_2(v) -
\de_\zz(v) = f(v)-\de_\zz(v)$ for $v \in V(L_2)-\{v_1,y\}$, and
 \begin{eqnarray*}
 d_T(x) &\le& (g_1(x)-1) = f(x), \\
 d_T(v_1) &\le& (g_1(v_1)-1) + (f_2(v_1)-\de_\zz(v_1)) =
	f(v_1)-\de_\zz(v_1), \\
 d_T(y) &\le& f_2(y)-\de_\zz(y) = f(y)-\de_\zz(y).
 \end{eqnarray*}
 Thus, $T$ is our desired spanning tree. We may henceforth assume that
$G_1$ and $g_1$ do not satisfy (\ref{TS}).

 Let $G_2^+ = L_2 \cup v_1 a_2 y$, where $a_2$ is a new vertex; $G_2^+$
is a minor of $G$ and hence is $K_4$-minor-free.  Extend $f_2$ to
$G_2^+$ by defining $f_2(a_2) = |V(G_2^+)|+1$.  Suppose that $G_2^+$ and
$f_2$ satisfy (\ref{TS}).
 Since $|V(B_1)| \ge 3$, we have
$|V(G_2^+)| < |V(G)|$, and so
 by induction there is a spanning $(f_2-\de_\zz)$-tree $Q_2$ of $G_2^+$.
 If $Q_2$ uses the path $v_1 a_2 y$ then $G_2$ has a spanning
$(f_2-\de_\zz)$-tree using $v_1 y$, which we have assumed does not
exist.  So $d_{Q_2}(a_2) = 1$ and deleting whichever of $v_1 a_2$, $a_2
y$ is an edge of $Q_2$ gives a spanning $(f_2-\de_zz)$-tree $Q_2'$ of
$L_2$ for which there is $w \in \{v_1, y\}$ such that $d_{Q_2'}(w) \le
f_2(w)-\de_\zz(w)-1$.  Let $w'$ be the vertex of $\{v_1,y\}$ other than
$w$.  Now $T_1' \cup Q_2' \cup xy$ is a connected spanning subgraph
containing a cycle through $x$, $v_1$ and $y$.  Let $w't$ be an edge of
this cycle incident with $w'$ and let $T = (T_1' \cup Q_2' \cup xy) -
w't$, which is a spanning tree in $G$.
 We have $d_T(v) \le
f_1(v) = f(v)$ for $v \in V(B_1)-\{x,v_1\}$, $d_T(v) \le f_2(v) -
\de_\zz(v) = f(v)-\de_\zz(v)$ for $v \in V(L_2)-\{v_1,y\}$, and
 \begin{eqnarray*}
 d_T(x) &\le& (f_1(x)-1)+1 = f(x), \\
 d_T(v_1) &\le& (f_1(v_1)-1) + (f_2(v_1)-\de_\zz(v_1))-1 =
	f(v_1)-\de_\zz(v_1), \\
 d_T(y) &\le& (f_2(y)-\de_\zz(y))+1-1 = f(y)-\de_\zz(y),
 \end{eqnarray*}
 where we add $1$ to the degrees of $x$ and $y$ due to $xy$, but
subtract $1$ from the degrees of $v_1$ and $y$ because $\{v_1, y\} =
\{w, w'\}$.
 Thus, $T$ is our desired spanning tree.  We may henceforth assume that
$G_2^+$ and $f_2$ do not satisfy (\ref{TS}).

 Now $G_1$ and $g_1$ do not satisfy (\ref{TS}), although $G_1$ and $f_1$
do.  Suppose (\ref{TS}) fails for $U \subseteq V(G_1)$.
 By definition of $g_1$, we see that $G_1$, $g_1$ and a specific set $S
\subseteq V(G_1)$ do satisfy (\ref{TS}) provided $v_1 \notin S$, $a_1
\in S$, or $\{x, v_1\} \subseteq S$ (since
$g_1(x)+g_1(v_1)=f_1(x)+f_1(v_1)$).
 Thus, $v_1 \in U$, $a_1 \notin U$ (so $U
\subseteq V(B_1)$), $x \notin U$, and
 $$c(B_1-U) = c(G_1-U) \ge \sum_{v \in U} (g_1(v)-1) + 1
	= \sum_{v \in U-\vv} (f(v)-1) + f(v_1)-f_2(v_1) +1.$$
 Also $G_2^+$ and $f_2$ do not satisfy (\ref{TS}), although $G_2$ and
$f_2$ do.  Suppose (\ref{TS}) fails for $W \subseteq V(G_2)$.
 We see that $G_2^+$, $f_2$ and a specific set $S \subseteq V(G_2)$ do
satisfy (\ref{TS}) provided $a_2
\in S$, or $a_2 \notin S$ and $c(G_2^+-S)=c(G_2-S)$.  The only
situation where $a_2 \notin S$ and $c(G_2^+-S)
\ne c(G_2-S)$ is when $\{v_1, y\} \subseteq S$.  Thus, $a_2 \notin W$ (so $W \subseteq V(L_2)$), $\{v_1, y\}
\subseteq W$, and
 $$c(L_2-W) = c(G_2^+-W)-1 \ge \sum_{v \in W} (f_2(v)-1)
	= \sum_{v \in W-\vv} (f(v)-1) + f_2(v_1)-1.$$
 Letting $S=U \cup W \subseteq V(G)$, we have $x \notin S$ but $\{v_1,
y\} \subseteq S$, and therefore
 \begin{eqnarray*}
  c(G-S) &=& c(B_1-U) + c(L_2-W) \\
	&\ge& \sum_{v \in S-\vv} (f(v)-1) + f(v_1) - f_2(v_1) + 1 +
			f_2(v_1) - 1 \\
	&=& \sum_{v \in S-\vv} (f(v)-1) + f(v_1)
	\;>\; \sum_{v \in S} (f(v)-1)
 \end{eqnarray*}
 which contradicts (\ref{TS}) for $G$ and $f$.  Thus, this final
situation never happens, and we can always find the required spanning
tree $T$.
 \end{proof}

\section{Concluding remarks}

In this section, we show the existence of planar graphs
 $G$ such that $c(G-S)\le \sum_{v\in S}(f(v)-1)$ for
every $S\subseteq V(G)$ and $S\ne \emptyset$, but where
$G$ has no spanning tree $T$ such that $d_T(v)\le f(v)$ for all $v\in
V(G)$, where there is an integer $f(v) \ge 2$ for each vertex.

Let $G$ be a $1$-tough planar graph with an integer $f(v) \ge 2$ for
each vertex, and let $G'$ be obtained  by attaching $f(v)-2$ pendant
edges to each $v \in V(G)$. Let $f(v)=2$ for $v\in V(G')-V(G)$.
 Since $c(G-S)\le |S|$ for every $S \subseteq V(G)$ with $S \ne \emptyset$, it
is not too hard to show that $c(G'-S')\le \sum_{v\in S'}(f(v)-1)$ for
every $S'\subseteq V(G')$ with $S'\ne \emptyset$.  If $G$ has no spanning
2-tree then $G'$ has no spanning tree $T$ such that $d_T(v)\le f(v)$ for
all $v\in V(G')$. Thus, it suffices to show the existence of a 1-tough
planar graph $G$ with no spanning 2-tree, i.e., no hamiltonian path.

A vertex in a graph is called a {\it simplicial vertex} if the neighbors
of this vertex form a clique in the graph. The graph $T$ given in
Figure~\ref{nhp}, constructed by Dillencourt~\cite{MR1115734} is
1-tough, non-hamiltonian, and maximal planar. Proposition 5 of that
paper states that any path in $T$ connecting any two of the three
vertices $A,B,C$ must omit at least one simplicial vertex.
In particular, $T$ has no hamiltonian path with both ends in $A, B, C$.

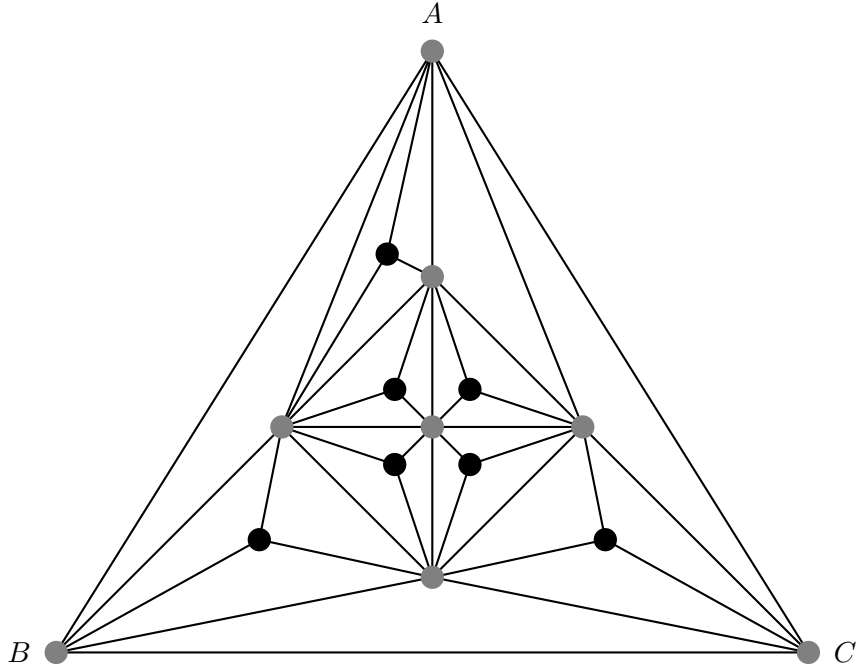
\begin{figure}[!htb]
\begin{center}
\begin{tikzpicture}[thick, scale=1]
 \tkzDefPoint(0,5){A}
    \tkzDefPoint(-5,-3){B}
    \tkzDefPoint(5,-3){C}
\draw [black, thick] (0,5)--(-5,-3)--(5,-3)--(0,5)--(0,2)--(0,0)--(0,-2)--
(-5,-3)--(-2.3,-1.5)--(-2,0)--(0,5)--(-0.6,2.3)--(0,2)--(-2,0)--(-0.5,0.5)--(0,2)
--(2,0)--(0.5,0.5)--(0,0)--(-2,0)--(0,-2)--(-0.5,-0.5)--(0,0)--(0.5,-0.5)--
(2,0)--(0,-2)--(2.3,-1.5)--(2,0)--(5,-3)--(0,-2);

\draw [black, thick] (-0.6,2.3)--(-2,0)--(-5,-3);
\draw [black, thick] (0,5)--(2,0);
\draw [black, thick] (-2.3,-1.5)--(0,-2);
\draw [black, thick] (-0.5,0.5)--(0,0)--(2,0);
\draw [black, thick] (-2,0)--(-0.5,-0.5);
\draw [black, thick] (0,2)--(0.5,0.5);
\draw [black, thick] (0,-2)--(0.5,-0.5);
\draw [black, thick] (5,-3)--(2.3,-1.5);

 % OLD VERSION:
 % \filldraw [gray] (0,5) circle (4pt);
 % \filldraw [gray] (-5,-3) circle (4pt);
 % \filldraw [gray] (5,-3) circle (4pt);
 % \filldraw [gray] (0,2) circle (4pt);
 % \filldraw [gray] (2,0) circle (4pt);
 % \filldraw [gray] (0,0) circle (4pt);
 % \filldraw [gray] (0,-2) circle (4pt);
 % \filldraw [gray] (-2,0) circle (4pt);

\draw [fill=white] (0,5) circle (4pt);
\draw [fill=white] (-5,-3) circle (4pt);
\draw [fill=white] (5,-3) circle (4pt);
\draw [fill=white] (0,2) circle (4pt);
\draw [fill=white] (2,0) circle (4pt);
\draw [fill=white] (0,0) circle (4pt);
\draw [fill=white] (0,-2) circle (4pt);
\draw [fill=white] (-2,0) circle (4pt);

\filldraw [black] (-2.3,-1.5) circle (4pt);
\filldraw [black] (-0.6,2.3) circle (4pt);
\filldraw [black] (-0.5,0.5) circle (4pt);
\filldraw [black] (0.5,0.5) circle (4pt);
\filldraw [black] (-0.5,-0.5) circle (4pt);
\filldraw [black] (0.5,-0.5) circle (4pt);
\filldraw [black] (2.3,-1.5) circle (4pt);

\node [] at (0,5.5) { $A$};
\node [] at (-5.3,-3) { $B\quad$};
\node [] at (5.3,-3) { $\quad C$};
\end{tikzpicture}
\end{center}
  \caption{{\small The 1-tough nonhamiltonian maximal planar graph $T$.
   The black vertices are the simplicial vertices.
}}\label{nhp}
\end{figure}

 Dillencourt constructed a sequence of 1-tough nonhamiltonian maximal
planar graphs as follows. Let $G_1=T$. For $n\ge 2$, let $G_n$ be
obtained from $G_{n-1}$ by deleting every simplicial vertex and
replacing it with a copy of $T$. More precisely, for each simplicial
vertex $u$, let $x,y,z$ be its neighbors. Delete $u$, insert a copy
$T_u$ (with vertices $A_u$, $B_u$, etc.) of $T$ inside the triangle
$(xyz)$, and add the edges $A_ux, A_uy, B_uy, B_uz, C_uz, C_ux$.

 Suppose that $G_n$, $n \ge 2$, has a hamiltonian path $P$.
 Since $G_{n-1}$ has at least three simplicial vertices,  for some such
vertex $u$ the copy $T_u$ of $T$ in $G_n$ contains neither end of $P$.
 The structure of $G_n$ guarantees that $P \cap T_u$ must have one of two
forms: (1) a single hamiltonian path in $T_u$ with ends being two of $A_u,
B_u, C_u$, or (2) a union of two paths $P_1$ and $P_2$, where $P_1$ is a
one-vertex path using one of $A_u, B_u, C_u$, and $P_2$ is a path between
the other two of $A_u, B_u, C_u$ using all other vertices of $T_u$.  In case
(2) we can join the vertex of $P_1$ to either end of $P_2$ to obtain a
hamiltonian path in $T_u$ with both ends in $A_u, B_u, C_u$.  So either
situation means $T$ has a hamiltonian path with ends being two of $A, B,
C$, which we know does not happen.  Therefore, $G_n$ is a $1$-tough
maximal planar graph with no hamiltonian path.

 \smallskip
 We conclude with a general question.
 Suppose $G$ is a tree, i.e., a connected graph of treewidth at most
$1$.
 Then the sufficient condition (\ref{suf-ftree}), the necessary
condition (\ref{nec-ftree}), and the even weaker condition
 \begin{equation}\label{weak-ftree}
 c(G-S) \;\;\le\;\; \sum_{v \in S} f(v)
	\quad\hbox{for all $S \subseteq V(G)$ with $S \ne \emptyset$}
 \end{equation}
 are all equivalent and all imply that $d_G(v) \le f(v)$ for
all $v \in V(G)$, i.e., that $G$ has (in fact, is) a spanning $f$-tree.
 (When $f(v)=k$ for all $v$, (\ref{weak-ftree}) just says that $G$ is
$\frac{1}{k}$-tough.)
 If $G$ has treewidth at most $2$, we have a sufficient condition
(\ref{TS}) that is only slightly stronger than the necessary condition
(\ref{nec-ftree}).  It therefore seems natural to ask the following.

 \begin{QUES}
 Is there a natural strengthening of Theorem \ref{ENV-ftree} involving
treewidth?  More specifically, suppose we have a connected graph $G$ of
treewidth $h$, and $f: V(G) \to \mZ$.
 Is there a condition involving $h$ and becoming stronger as $h$
increases, similar to (\ref{nec-ftree}) or (\ref{TS}) but weaker than
(\ref{suf-ftree}), that implies that $G$ has a spanning $f$-tree?
 \end{QUES}

 \section*{Acknowledgement}

 The authors thank an anonymous referee who pointed out a problem with
an earlier version of Proposition \ref{com-counting}.

\bibliographystyle{plain}
%\bibliography{SSL-BIB}

\begin{thebibliography}{10}

\bibitem{Tough-CounterE}
D.~Bauer, H.~J. Broersma, and H.~J. Veldman.
\newblock Not every 2-tough graph is {H}amiltonian.
\newblock In {\em Proceedings of the 5th {T}wente {W}orkshop on {G}raphs and
  {C}ombinatorial {O}ptimization ({E}nschede, 1997)}, volume~99, pages
  317--321, 2000.

\bibitem{STGT-ch6}
J.-C. Bermond.
\newblock Hamiltonian graphs.
\newblock In {\em Selected topics in graph theory}, pages 127--167. Academic
  Press, London and New York, 1978.

\bibitem{chvatal-tough-c}
V.~Chv{\'a}tal.
\newblock Tough graphs and {H}amiltonian circuits.
\newblock {\em Discrete Math.}, 5:215--228, 1973.

\bibitem{MR1115734}
Michael~B. Dillencourt.
\newblock An upper bound on the shortness exponent of {$1$}-tough, maximal
  planar graphs.
\newblock {\em Discrete Math.}, 90(1):93--97, 1991.

\bibitem{MR0175809}
R.~J. Duffin.
\newblock Topology of series-parallel networks.
\newblock {\em J. Math. Anal. Appl.}, 10:303--318, 1965.

\bibitem{K4MF-no-2-walk}
Zden{\v{e}}k Dvo{\v{r}}{\'a}k, Daniel Kr{\'a}l', and Jakub Teska.
\newblock Toughness threshold for the existence of 2-walks in
  {$K_4$}-minor-free graphs.
\newblock {\em Discrete Math.}, 310(3):642--651, 2010.

 \bibitem{ENV}
 M. N. Ellingham, Yunsun Nam and Heinz-J\"{u}rgen Voss,
Connected $(g,f)$-factors.
 {\em J. Graph Theory}, 39(1): 62--75, 2002.

\bibitem{MR785651}
Hikoe Enomoto, Bill Jackson, P.~Katerinis, and Akira Saito.
\newblock Toughness and the existence of {$k$}-factors.
\newblock {\em J. Graph Theory}, 9(1):87--95, 1985.

 \bibitem{Ha17v5}
 Morteza Hasanvand.
 Spanning trees and spanning Eulerian subgraphs with small degrees. II.
 \texttt{arXiv:1702.06203v5}, January 2018.

\bibitem{JW-k-walks}
B.~Jackson and N.~C. Wormald.
\newblock {$k$}-walks of graphs.
\newblock {\em Australas. J. Combin.}, 2:135--146, 1990.
\newblock Combinatorial mathematics and combinatorial computing, Vol. 2
  (Brisbane, 1989).

 %

 \bibitem{KKS09}
 Mikio Kano, Gyula Y. Katona and J\'{a}cint Szab\'{o}.
 Elementary graphs with respect to $f$-parity factors.
 \emph{Graphs Combin.} 25(5):717-726, 2009.

 \bibitem{WC-part2tree}
 Joseph A. Wald and Charles J. Colbourn, Steiner trees,
partial 2-trees, and minimum IFI networks.
 {\em Networks} 13(2):159--167, 1983.

\bibitem{Win-tough}
Sein Win.
\newblock On a connection between the existence of {$k$}-trees and the
  toughness of a graph.
\newblock {\em Graphs Combin.}, 5(2):201--205, 1989.

\end{thebibliography}

\newdimen\pri\pri=\parindent % rotten LaTeX redefines \parindent
\def\bibitemx#1{\smallbreak\leavevmode%
	\hangindent\parindent\noindent\hbox{}\kern-\pri[#1] }
\let\newblock\relax

\end{document}